\documentclass[12pt,twoside]{article}

\usepackage{graphicx,graphics, amsthm, amsfonts, amsbsy,amssymb, amsmath, cite, array,arydshln, hyperref, float,tikz,booktabs}

\usetikzlibrary{positioning}

\everymath{\displaystyle}

\usepackage[margin=1in]{geometry}
\renewenvironment{proof}{{\noindent\bfseries Proof.}}{\qed}
\allowdisplaybreaks

\let\OLDthebibliography\thebibliography
\renewcommand\thebibliography[1]{
	\OLDthebibliography{#1}
	\setlength{\parskip}{4pt}
	\setlength{\itemsep}{1pt plus 0.9ex}
}

\newtheorem{theorem}{Theorem}[section]

\newtheorem{example}[theorem]{Example}

\newtheorem{lemma}[theorem]{Lemma}
\newtheorem{corollary}[theorem]{Corollary}

\def\qed{\nolinebreak\hfill\rule{.2cm}{.2cm}\par\addvspace{.5cm}}

\begin{document}
	\title{Spectral Properties of Zero-Divisor Graphs of Truncated Polynomial Rings}
	\author{{\small Bilal Ahmad Rather}\\[2mm]
		{\small School of Mathematics and Statistics, Shandong University of Technology, Zibo 255049, China}\\
		\small \texttt{bilahamadrr@gmail.com}
					}
	\date{}
	
	\pagestyle{myheadings} \markboth{Bilal Ahmad Rather}{Spectral Properties of Zero-Divisor Graphs of Truncated Polynomial Rings}
	\maketitle
	\begin{abstract}
		Let $R$ be a commutative ring with identity and let $Z^{\ast}(R)$ denote the set of nonzero zero-divisors of $R$.
		The \emph{zero-divisor graph} $ \varGamma(R)$ is the simple graph with vertex set $V( \varGamma(R))=Z^{\ast}(R)$, where two distinct vertices$x,y\in Z^{\ast}(R)$ are adjacent if and only if $xy=0$ in $R$. In this paper we investigate the zero-divisor graph of the truncated polynomial ring $R=\mathbb{Z}_{p}[x]/\langle x^{c}\rangle,$ for $c\in\mathbb{N}.$ We determine the spectrum of the $A_{\alpha}$-matrix associated with $ \varGamma(R)$, and, as special cases, explicitly obtain both the adjacency spectrum and the signless Laplacian spectrum of $ \varGamma(R)$. Furthermore, we prove that the Laplacian eigenvalues, as well as the distance eigenvalues, of these graphs are all integers.
	\end{abstract}
	
	
	
	\section{Introduction}
	
	A (finite, simple) graph $G=(V(G),E(G))$ consists of a finite vertex set $V(G)$ and an edge set
	$E(G)\subseteq \{\{u,v\}:u,v\in V(G), u\neq v\}$. The \emph{order} and \emph{size} of $G$ are
	$n=|V(G)|$ and $m=|E(G)|$, respectively. For vertices $u,v\in V(G)$, we write $u\sim v,$ if
	$\{u,v\}\in E(G)$. The \emph{open neighbourhood} of a vertex $v$ is $N_G(v) = \{u\in V(G):u\sim v\},$
	and the \emph{degree} of $v$ is $d_G(v)=|N_G(v)|$ (or simply $d(v)$ if the graph is clear).
	An \emph{independent set} in $G$ is a subset of vertices no two of which are adjacent. A
	\emph{clique} is a vertex set $X\subseteq V(G)$ such that the induced subgraph $G[X]$ is
	complete. A vertex is universal if it is adjacent to all other vertices of $G$. The complete graph on $n$ vertices is denoted by $K_n$. For other notations, see \cite{cds}.
	
	The \emph{adjacency matrix} of $G$ is the $n\times n$ matrix $A(G)=(a_{ij}),$ where
	\[
	a_{ij} =
	\begin{cases}
		1, & \text{if } v_i\sim v_j,\\[2pt]
		0, & \text{otherwise}.
	\end{cases}
	\]
	The matrix $A(G)$ is real symmetric, hence all its eigenvalues are real. We list them in
	non-increasing order as $\lambda_1(G)\geq \lambda_2(G)\geq \dots \geq \lambda_n(G).$
	The multiset $\{\lambda_1(G),\dots,\lambda_n(G)\}$ is called the \emph{adjacency spectrum} (or spectra) of $G$.
	The \emph{characteristic polynomial} of a real matrix $M$ is $\Phi(M,\lambda) = \det(\lambda I - M),$
	where $I$ is the identity matrix. The \emph{energy} of a graph $G$, introduced by Gutman
	in the context of theoretical chemistry, is defined by
	\[
	\mathcal{E}(G) = \sum_{i=1}^n |\lambda_i(G)|.
	\]
	The energy of $G$ has been extensively studied, see
	\cite{gutman2010,filipovski,shi,nv} and the references therein. Let $D(G)$ be the diagonal matrix of vertex degrees, that is, $D(G) = \operatorname{diag}(d(v_1),\dots,d(v_n)).$ 	The \emph{Laplacian matrix} of $G$ is defined as $	L(G) = D(G)-A(G),$	while the \emph{signless Laplacian} is defined by $Q(G) = D(G)+A(G).$
	Both $L(G)$ and $Q(G)$ are real symmetric and positive semidefinite (definite) matrices. Their spectra
	encode a number of combinatorial properties of the underlying graph, see
	\cite{cds,brouwer-haemers}.
	
	Motivated by unifying the adjacency and signless Laplacian theories, Nikiforov	\cite{nikiforov-aalpha}, introduced the one–parameter family of matrices
	\[
	A_{\alpha}(G) = \alpha D(G) + (1-\alpha)A(G),\qquad 0\leq \alpha\leq 1.
	\]
	For some special values of $\alpha$, $A_\alpha(G)$ matrix reduces to classical matrices, like
	\[
	A_0(G) = A(G),\qquad
	A_1(G) = D(G),\qquad
	2A_{1/2}(G) = D(G)+A(G)=Q(G).
	\]
	Moreover, for any distinct $\alpha,\beta\in[0,1]$ one has
	\[
	\frac{1}{\alpha-\beta}\bigl(A_\alpha(G)-A_\beta(G)\bigr)
	= D(G)-A(G)=L(G).
	\]
	So, the adjacency, the Laplacian and the  Laplacian matrices can  be recovered from the $A_\alpha$–family. Nikiforov \cite{nikiforov-aalpha}, provided basic properties of $A_\alpha$ matrix, like positive definiteness, inequalities for the spectral radius (the largest eigenvalue), and spectra of various graph families. The $A_{\alpha}$ spectra of power graphs, energy and other properties are given in \cite{bilalaims,bilalakcej,bilaljaa}. Nikiforov et. al. \cite{NikiforovTrees} presented $A_{\alpha}$ spectra of trees, Alhevaz \cite{AAlphaBounds} obtained the inequalities for the spectral radius of the $A_{\alpha}$ matrix, and Liu, Xue and Shu \cite{LinXueShu2018} found results on $A_\alpha$ spectral radius and related invariants.

	The \emph{distance matrix} of a connected graph $G$ is the $n\times n$ matrix $D_{\mathrm{dist}}(G)$
	whose $(i,j)$–entry is the graph distance $d_G(v_i,v_j)$. Let
	\[
	\operatorname{Tr}(G) = \operatorname{diag}(t(v_1),\dots,t(v_n)),\qquad
	t(v_i) = \sum_{j=1}^n d_G(v_i,v_j),
	\]
	be the diagonal matrix of vertex transmissions. Following Aouchiche and Hansen
	\cite{aouchiche-hansen-distance}, the \emph{distance Laplacian} of $G$ is defined by $\mathcal{L}(G) = \operatorname{Tr}(G) - D_{\mathrm{dist}}(G).$
	The spectral properties of $\mathcal{L}(G)$ form an analogue of Laplacian spectral theory
	for distance–based matrices and have attracted considerable attention in recent years
	\cite{aouchiche-hansen-distance,Ah2013}. A detailed survey can be seen in \cite{AH3,bilalDlsurvey}.
	
	A real square matrix $M$ is called \emph{integral} if all its eigenvalues are integers. A graph
	$G$ is said to be \emph{integral} (with respect to a given graph matrix) if the eigenvalues of
	that matrix are integers. Integral spectra arise naturally in several contexts: they appear in
	the study of combinatorial designs and strongly regular graphs, and they are related to algebraic
	graph theory and number–theoretic questions, see, for example,
	\cite{cds}. Harary and Schwenk first studied the problem of the integral graphs \cite{harary-schwenk}. Klotz and Sander \cite{klotz-sander} presented properties of integral graphs. Some circulant integral graphs can be seen in \cite{so-integral}.   Integral Laplacian or
	distance Laplacian spectra are particularly attractive, since they combine structural graph
	information with arithmetic constraints and often lead to strong classification results. If a graph has an eigenvalue $\lambda$ with multiplicity $k$, then we represent it by $\lambda^{[k]}$. In the present work, we are interested in graphs for which various matrices are integral, and we determine their spectra explicitly for a family of
	zero-divisor graphs.

	Let $R$ be a commutative ring with identity $1\neq 0$. A nonzero element $x\in R$ is a
	\emph{zero-divisor}, if there exists a nonzero $y\in R$ such that $xy=0$. Beck
	\cite{ib} introduced the notion of the \emph{zero–divisor graph} of $R$, including $0$ as a
	vertex. Later, Anderson and Livingston \cite{al} modified the definition by considering only
	nonzero zero-divisors as vertices and joining distinct vertices $x,y$ whenever $xy=0$. Unless
	stated otherwise, we adopt the latter convention here. The adjacency spectrum of zero–divisor graphs was first considered by Young \cite{my}, and has since been studied in a variety of settings, see, for instance,
	\cite{bilalijpam,monius,fareeha,bilalcoo,bilalejgta,johnson} and the references therein.
	Beyond adjacency spectra, Laplacian and signless Laplacian spectra, as well as other spectral
	invariants and topological indices of zero–divisor graphs, have also been investigated. 
	
	The paper is organized as: Section \ref{section 2} gives the $A_{\alpha}$ spectra of zero-divisor graph $\mathbb{Z}_p[x]/\langle x^{2b}\rangle $ where $b\in \mathbb{N}$ is a positive integer. In Section \ref{section 3}, we give the structure properties of $\mathbb{Z}_p[x]/\langle x^{2b+1}\rangle,$ and  present its $A_{\alpha}$ spectra. We end up the article with comments showing that $ \varGamma(\mathbb{Z}_p[x]/\langle x^{c}\rangle), $ with $c\in \mathbb{N}$ is distance Laplacian integral.

	\section{Spectra of zero-divisor graph of $\mathbb{Z}_p[x]/\langle x^{2b}\rangle $}\label{section 2}
	Let $p$ be a prime and let $c\in \mathbb{N}$ be a positive integer. We consider the quotient ring $R = \mathbb{Z}_p[x]/\langle x^{c}\rangle.$	The ring $R$ is a finite commutative local ring with maximal ideal $\mathfrak{m}=\langle x\rangle$ and residue field $R/\mathfrak{m}\cong \mathbb{Z}_p$. As a $\mathbb{Z}_p$-vector space, $R$ has basis $\{1,x,\dots,x^{c-1}\}$, and hence $|R| = p^{c}$. Moreover, $\mathfrak{m}$ is nilpotent with $\mathfrak{m}^{c}=\langle x^{c}\rangle=0$. So, $R$ is an Artinian principal ideal ring of length $c$ which is not reduced. The following result gives the structural properties of the zero-divisor graph of ring  $\mathbb{Z}_p[x]/\langle x^{2b}\rangle$.
	\begin{theorem}\label{even case}
		Let $p$ be a prime and let $a=2b$, where $b\in \mathbb{N}$. Let $R = \mathbb{Z}_p[x]/\langle x^{2b}\rangle\cong\mathbb{F}_p[x]/(x^{2b}),$ and  $ \varGamma(R)$ be its zero-divisor graph. Then the following hold.
		\begin{enumerate} 
			\item The order of $ \varGamma(R)$ is $|V( \varGamma(R))|= |Z(R)\setminus\{0\}|= p^{ 2b-1} - 1$, and its vertex set can be written as $V( \varGamma(R))=L\cup U,$ where $L=\bigcup_{k=1}^{b-1} V_k, ~U = \bigcup_{k=b}^{2b-1} V_k,$ and $|V_i|= (p-1)p^{ 2b-1-i},  $ for $ i=1,\dots,2b-1.$
			\item In $ \varGamma(R)$, $C=\bigcup_{k=b}^{2b-1} V_k$ is a clique of size $p^{b} - 1$, and $L=\bigcup_{k=1}^{b-1}V_k$ is an independent set of size $p^{2b-1}-p^{b}$.
			\item For $v\in V_i$, its neighbors are exactly the vertices in $\bigcup_{j: i+j\ge 2b} V_j$, and if $i\ge b$, then $v$ is also adjacent to all other vertices in $V_i$.
			\item The diameter of $ \varGamma(R)$ is $2$, and its girth is $3.$
		\end{enumerate}
	\end{theorem}
	\begin{proof}
		With the description of $R$ as an $\mathbb{F}_p$--vector space, the characterization of
		units by the constant term, and the identification of nonzero zero-divisors with the
		nonzero elements of $(x)$ are standard for $\mathbb{F}_p[x]/(x^{2b})$. The $\mathbb{F}_p$--vector space with multiplication modulo $x^{2b}$ is
		\[
		R = \Bigl\{\sum_{k=0}^{2b-1} a_k x^k : a_k \in \mathbb{F}_p\Bigr\}.
		\]
		 So, all terms of degree greater or equal to $2b$ vanish. An element $\sum_{k=0}^{2b-1} a_k x^k$ is a unit if and only if $a_0\neq 0$.
		Since $R$ is finite, every nonunit is a zero-divisor, denoted by $Z(R)$. Thus, the set of nonzero zero-divisors is
		\[
		Z(R)\setminus\{0\}
		= \Bigl\{\sum_{k=1}^{2b-1} a_k x^k : (a_1,\dots,a_{2b-1}) \neq (0,\dots,0)\Bigr\},
		\]
		which are precisely the the nonzero elements of the ideal $(x)$, and hence we have
		\[
		|V( \varGamma(R))|
		= |Z(R)\setminus\{0\}|
		= p^{ 2b-1} - 1.
		\]
		For (1), every nonzero element $f$ of $(x)$ has a smallest degree $k$ with nonzero coefficient,
		denoted by $\operatorname{mindeg}(f)$. For a nonzero zero-divisor $f \in Z(R)$, define the minimal degree
		\[
		\operatorname{mindeg}(f)
		= \min\bigl\{k \in \{1,\dots,2b-1\} : \text{the coefficient of $x^k$ in $f$ is nonzero}\bigr\}.
		\]
		Fixing $a_i\in\mathbb{F}_p^\times$ and allowing the
		higher coefficients to vary freely gives cardinalities of the sets $V_i$.  This partitions the vertex set $V( \varGamma(R)) = Z(R)\setminus\{0\}$ into
		\[
		V( \varGamma(R)) = V_1 \cup V_2 \cup \dots\cup V_{b-1}\cup V_{b}\cup V_{b+1}\cup \dots \cup V_{2b-1},
		\]
		where
		\[
		V_i
		= \left\{
		\sum_{k=i}^{2b-1} a_k x^k : a_i \neq 0
		\right\}, \qquad 1\le i \le 2b-1.
		\]
		For each $i$, the coefficient $a_i$ can be any nonzero element of $\mathbb{F}_p$,
		and the coefficients $a_{i+1},\dots,a_{2b-1}$ are arbitrary.  Hence, we obtain $|V_i|= (p-1)p^{ 2b-1-i}, $ for $ i=1,\dots,2b-1,$
		and
		\[
		\sum_{i=1}^{2b-1} |V_i|
		= (p-1)\sum_{t=0}^{2b-2} p^t
		= p^{ 2b-1}-1
		= |V( \varGamma(R))|.
		\]
		We decompose $ V( \varGamma(R))$ into subsets $	L = \bigcup_{k=1}^{b-1} V_k,  $ and $
		U = \bigcup_{k=b}^{2b-1} V_k,$ so that $V( \varGamma(R))=L\cup U.$\\
		For (2), let $f,g\in Z(R)\setminus\{0\}$ with $\operatorname{mindeg}(f)=i$ and
		$\operatorname{mindeg}(g)=j$, where $1\le i,j\le 2b-1$. Consider 
		\[
		f = a_i x^i + \text{(higher degree terms)}, ~\text{and}~
		g = b_j x^j + \text{(higher degree terms)},
		\]
		where $a_i,b_j \in \mathbb{F}_p^\times$.
		In product $fg$, the smallest degree term is $a_i b_j x^{i+j}$ $x^{i+j}$ with coefficient $a_i b_j$, as all other products of coefficients have strictly larger degree. Thus $fg = 0  $ if and only if $ i+j \ge 2b.$ Indeed, if $i+j \le 2b-1$, then the $x^{i+j}$ term is nonzero in $R$. If $i+j\ge 2b$, then every term of $fg$ has degree $\ge 2b$ and vanishes modulo $(x^{2b})$. Consequently, adjacency in $ \varGamma(R)$ is completely determined by the minimal degrees, that is, $f \in V_i,g\in V_j, $ with $ f\neq g$, then $f $ is adjacent to $g $ if and only if $ i+j \ge 2b.$ Thus, it follows that, for each $i$, two distinct vertices in $V_i$ are adjacent if and only if $2i\ge 2b$, that is, $i\ge b$. Thus, $V_i$ is an independent set, if $i<b,$ and a clique, if $i\ge b.$	In particular, $L=\bigcup_{k=1}^{b-1}V_k$ is a disjoint union of independent subsets. For $i\neq j$, all edges between $V_i$ and $V_j$ either exist or do not exist simultaneously. In particular, each vertex of $V_{i}$  is adjacent to $V_j$, if and only if $i+j\geq 2b.$ We note that  if $1\le i,j\le b-1$, then $i+j\le 2b-2<2b$, so there are no edges at all inside $L$. On the other hand, for all $i,j\ge b$, we have $i+j\ge 2b$. So, every pair of vertices in $	C = \bigcup_{k=b}^{2b-1} V_k = V_b \cup U$		is adjacent. Thus $C$ induces a clique, while $L$ is an independent set.
		
		For (3), with $v\in V_i$, its neighbors are exactly the vertices in $\bigcup_{j:i+j\ge 2b} V_j$, that is, each vertex of $V_{i}$ is adjacent to every $V_{j}$, if and only if $i+j\geq 2b.$ Also, for $v\in V_{i}$,  if $2i\ge 2b$, that is,  $i\ge b$, then $v$ is also adjacent to all other vertices in $V_i$. 
		Using
		\[
		\sum_{j=m}^{2b-1} |V_j| = p^{ 2b-m}-1,
		\quad\text{for }1\le m\le 2b-1,
		\]
		and with $m=2b-i$, we get the total number of vertices in subsets $V_{2b-i},\dots,V_{2b-1}$ as $\sum_{j=2b-i}^{2b-1} |V_j| = p^{ i}-1.$
		This includes $V_i$ itself, provided $2i\ge 2b$, or $i\ge b$. Thus, for $v\in V_{i}, $ we have
		\[
		\deg(v) =
		\begin{cases}
			p^{ i}-1, & \text{if } i<b,\\
			p^{ i}-2, & \text{if } i\ge b.
		\end{cases}
		\]
		In particular, for $i=2b-1$, it is clear that $i\ge b$, so each vertex in $V_{2b-1}$ is adjacent to every $V_{i}$ in $ \varGamma(R)$, and degree of $v\in V_{2b-1}$	 is  $\deg(v) = p^{ 2b-1} - 2.$ Thus $V_{2b-1}$ consists of $p-1$ universal vertices. For the middle subset with $i=b$, each vertex in $V_b$ has degree $\deg(v) = p^{ b} - 2, $ for $ v\in V_b,$ and is adjacent to all vertices in $C$ except itself.
		
		As each $V_{i}$ induces a clique if $i\geq b$ and each vertex of $V_{i}$ is adjacent to every vertex of $V_{i}$, since $i+j\geq 2b$ for $b\leq i\leq j\leq2b-1. $ Thus, the union of $V_{i}$'s is a clique $C = \bigcup_{k=b}^{2b-1} V_k$, and its size is
		\[
		|C|
		= \sum_{k=b}^{2b-1} |V_k|
		= \sum_{k=b}^{2b-1} (p-1)p^{ 2b-1-k}
		= p^{ 2b-b} - 1
		= p^{b} - 1.
		\]
		We claim that the clique number of $ \varGamma(R)$ is $\omega( \varGamma(R)) = p^b-1$. Let $K$ be any clique in $ \varGamma(R)$, and write
		\[
		i_0 = \min\{\operatorname{mindeg}(v) : v\in K\}.
		\]
		Then $K\subseteq \bigcup_{j\ge i_0} V_j$, and $K$ contains some vertex in $V_{i_0}$.
		If $w\in V_j$ lies in $K$, adjacency with a vertex in $V_{i_0}$ forces $i_0+j\ge 2b$, 	that is, $j\ge 2b-i_0$. Thus every vertex of $K$ lies in $\bigcup_{j=m}^{2b-1} V_j, $ with $ m = \max\{i_0,2b-i_0\}.$ Hence, it follows that
		\[
		|K|
		\le \sum_{j=m}^{2b-1} |V_j|
		= p^{ 2b-m}-1.
		\]
		It is clear that $m\ge b$ for all $1\le i_0\le 2b-1$, with equality $m=b$, only
		when $i_0=b$. Thus, $|K|\le p^{ 2b-b}-1 = p^b-1,$	and since $C$ has size $p^b-1$, and we obtain $\omega( \varGamma(R)) = p^b - 1.$ In particular, when $b=1$ (that is, $a=2$), $ \varGamma(R)\cong K_{p-1}$ and $\omega( \varGamma(R))=p-1$. Furthermore, $L=\bigcup_{k=1}^{b-1}V_k$ is an independent set, so $V( \varGamma(R)) = L \cup C$. For each $i<b$, the subset $V_i$ is an independent set of size	$|V_i| = (p-1)p^{ 2b-1-i}.$ In particular, $V_1$ is an independent set of size $|V_1|=(p-1)p^{ 2b-2}.$ Since, all subsets $V_i$ with $1\le i\le b-1$ are mutually nonadjacent, their union $L = \bigcup_{k=1}^{b-1} V_k$ is an independent set of size
		\[
		|L|
		= \sum_{k=1}^{b-1} |V_k|
		= \bigl(p^{ 2b-1}-1\bigr) - \sum_{k=b}^{2b-1} |V_k|
		= \bigl(p^{ 2b-1}-1\bigr) - \bigl(p^b-1\bigr)
		= p^{ 2b-1} - p^b.
		\]
		This gives a large explicit independent set complementary to the maximal clique $C$.\\
		(4) Since $V_{2b-1}$ is nonempty and consists of $p-1$ universal vertices, the graph $ \varGamma(R)$ is connected for all $b\ge1$. Moreover, if $b=1$ (so $a=2$), then $V_1$ is a clique and $ \varGamma(R)\cong K_{p-1}$. The diameter of $ \varGamma(R)$ is $1$ when $p\ge3$. For $p=2$, there is a single vertex and the diameter of $ \varGamma(R)$ is $0$.  If $b\ge 2$, so $a=2b\ge4$, and  there exist nonadjacent pairs, but any two vertices are at distance at most $2$ due to universal vertices present in $V_{2b-1}$. Thus, for $b\ge2$, the diameter of $ \varGamma(R)$ is $2$. For $b\ge 2$ ($a=2b\ge4$), consider the subsets $V_{2b-2}$ and $V_{2b-1}$, we get $2(2b-2)\ge 2b, $ and $ (2b-2)+(2b-1)\ge 2b$. So, $V_{2b-2}$ is a clique and every vertex in $V_{2b-2}$ is adjacent to every vertex in $V_{2b-1}$. Since $|V_{2b-2}|\ge 2$ and $|V_{2b-1}|\ge 1$, it yields triangles, so for all primes $p$ and all $b\ge2$ the girth of $ \varGamma(R)$ is $3$. For $b=1$, $ \varGamma(R)\cong K_{p-1}$, and the girth is $3$ if $p-1\ge3$.
	\end{proof}
	
	We will illustrate Theorem \ref{even case} for $b=3$.	Let $p$ be a prime and let $R=\mathbb{Z}_p[x]/\langle x^6\rangle \cong \mathbb{F}_p[x]/(x^6).$
	Then the zero-divisor graph  $ \varGamma(R)$  as an $\mathbb{F}_p$--vector space, can be written as
	\[
	R = \{a_0 + a_1x + a_2x^2 + a_3x^3 + a_4x^4 + a_5x^5 : a_i \in \mathbb{F}_p\},
	\]
	with multiplication modulo $x^6$. So, all terms of degree greater or equal to $6$ vanish. An element of $R$ is a unit if and only if $a_0 \neq 0$, hence the nonzero zero-divisors are
	\[
	Z(R)\setminus\{0\}
	= \{a_1x + a_2x^2 + a_3x^3 + a_4x^4 + a_5x^5 : (a_1,a_2,a_3,a_4)\neq (0,0,0,0)\}.
	\]
	Thus, $ \varGamma(R)$ has $p^5 - 1$ vertices.  For a nonzero zero-divisor $f\in Z(R)$, define the minimal degree
	\[
	\operatorname{mindeg}(f)
	= \min\{k \in \{1,2,3,4,5\} : \text{the coefficient of } x^k \text{ in } f \text{ is nonzero}\}.
	\]
	This partitions $V( \varGamma(R))=Z(R)\setminus\{0\}$ into following disjoint subsets as:
	\[
	\begin{aligned}
		V_1 &= \{a_1x + a_2x^2 + a_3x^3 + a_4x^4 + a_5x^5 : a_1 \neq 0\},\\
		V_2 &= \{a_2x^2 + a_3x^3 + a_4x^4 + a_5x^5 : a_2 \neq 0\},	V_3 = \{a_3x^3 + a_4x^4 + a_5x^5 : a_3 \neq 0\},\\
		V_4 &= \{a_4x^4 + a_5x^5 : a_4 \neq 0\}, \quad V_5 = \{a_5x^5 : a_5 \neq 0\}.
	\end{aligned}
	\]
	The cardinalities of above vertex sets are $ |V_1|= (p-1)p^4 = p^5 - p^4, |V_2| = (p-1)p^3 = p^4 - p^3, |V_3| = (p-1)p^2 = p^3 - p^2, |V_4| = (p-1)p   = p^2 - p$ and $ |V_5| = p-1$
	Thus, $V( \varGamma(R)) = V_1 \cup V_2 \cup V_3 \cup V_4 \cup V_5,$ and $|V( \varGamma(R))|=|V_1|+|V_2|+|V_3|+|V_4|+|V_5| = p^5 - 1.$  Let $f,g\in Z(R)\setminus\{0\}$ with $\operatorname{mindeg}(f)=i$ and
	$\operatorname{mindeg}(g)=j$, where $i,j\in\{1,2,3,4,5\}$. Then $fg = 0$  if and only if  $i+j \ge 6.$
	In particular, the adjacency relations in $ \varGamma(R)$ are: if $i=j=1$, then $1+1=2<6$ implies that no edges within $V_1$. So, $V_1$ is an independent set. If  $i=j=2$, then  $2+2=4<6$ implies that there are no edges within $V_2$. So, $V_2$ is an independent set. If  $i=j=3$, then $3+3=6\ge 6$, and it follows that all pairs in $V_3$ are adjacent. Thus, $V_3$ induces a clique.  In general, the adjacency relations are: $V_1$ and $V_2$ are independent sets, $V_3$, $V_4$, and $V_5$ are cliques. There  are no edges between $V_1$ and $V_2$, $V_1$ and $V_3$, or $V_1$ and $V_4$. Each vertex of $V_1$ is completely joined to $V_5$, each vertex of $V_2$ is completely joined to $V_4\cup V_5$ but not to $V_3$.  The induced subgraph on $V_3\cup V_4\cup V_5$ is complete $ K_{|V_3|+|V_4|+|V_5|}= K_{p^3-1}. $ The degree sequence of the vertices in $V_{i}$ are:
	\[
	\deg(v) =
	\begin{cases}
		|V_5| = p-1,
		& v\in V_1,\\[4pt]
		|V_4| + |V_5|
		= (p^2-p) + (p-1) = p^2 - 1,
		& v\in V_2,\\[4pt]
		(|V_3|-1) + |V_4| + |V_5|
		= (p^3-p^2-1) + (p^2-p) + (p-1) = p^3 - 2,
		& v\in V_3,\\[4pt]
		|V_2| + |V_3| + (|V_4|-1) + |V_5|
		= (p^4-p^3) + (p^3-p^2)\\
		+ (p^2-p-1) + (p-1)
		= p^4 - 2,
		& v\in V_4,\\[4pt]
		(p^5 - 1) - 1 = p^5 - 2,
		& v\in V_5.
	\end{cases}
	\]
	In particular, each vertex in $V_5$ is adjacent to every other vertex of $ \varGamma(R)$. So, $ \varGamma(R)$ is connected. Moreover, $V_3\cup V_4\cup V_5$ induces a clique
	of size $p^3-1 \ge 3$, so $ \varGamma(R)$ contains triangles and its girth is $3$.\\
	A particular example with $p=2$ is given below:\\
 For $p=2$, $R = \mathbb{F}_2[x]/(x^6), $ and every  element of $R$ can be written uniquely as
 \[
 a_0 + a_1x + a_2x^2 + a_3x^3 + a_4x^4 + a_5x^5,
 \quad a_i \in \mathbb{F}_2 = \{0,1\},
 \]
 with multiplication modulo $x^6$. So, $x^6 = 0$ and all terms of degree greater or equal to $6$ vanish. 	An element is a unit if and only if $a_0 \neq 0$. Thus, the nonzero zero-divisors are precisely the nonzero multiples of $x$
 \[
 Z(R)\setminus\{0\} 
 = \{a_1x + a_2x^2 + a_3x^3 + a_4x^4 + a_5x^5 : (a_1,a_2,a_3,a_4,a_5) \neq (0,0,0,0,0)\}.
 \]
 So, the zero-divisor graph $ \varGamma(R)$ has  $|V( \varGamma(R))| = 2^5 - 1 = 31$ vertices. For a nonzero $f\in Z(R)$, we define
 \[
 \operatorname{mindeg}(f) = \min\{k\in\{1,2,3,4,5\} : \text{the coefficient of }x^k\text{ in }f\text{ is }1\}.
 \]
 This partitions $V( \varGamma(R))$ into $V_1,\dots,V_5$. For $V_1$, minimal degree is $1$, that is the first nonzero term is $x$. Here $a_1 = 1$, and $a_2,a_3,a_4,a_5 \in \{0,1\}$ arbitrary, so $|V_1| = 2^4 = 16.$ The elements in $V_{1}$ are
 \begin{align*}
 	V_1 = \{&
 	x,
 	x + x^2,
 	x + x^3,
 	x + x^4,
 	x + x^5,x + x^2 + x^3,
 	x + x^2 + x^4,
 	x + x^2 + x^5,
 	x + x^3 + x^4,\\
 	&x + x^3 + x^5,x + x^4 + x^5,
 	x + x^2 + x^3 + x^4,
 	x + x^2 + x^3 + x^5,x + x^2 + x^4 + x^5,\\
 	&x + x^3 + x^4 + x^5,x + x^2 + x^3 + x^4 + x^5
 	\}.
 \end{align*}
 For $V_2$, the minimal degree is $2$, so the initial nonzero term is $x^2$.	Thus, $a_1 = 0$, $a_2 = 1$, and $a_3,a_4,a_5 \in \{0,1\}$ arbitrary, so $|V_2| = 2^3 = 8.$ Clearly, $V_{2}$ has elements
 \begin{align*}
 	V_2 = \{& x^2, 	x^2 + x^3, x^2 + x^4, 	x^2 + x^5,x^2 + x^3 + x^4, x^2 + x^3 + x^5, x^2 + x^4 + x^5,\\
 	&x^2 + x^3 + x^4 + x^5 	\}.
 \end{align*}
 For	$V_3$, the minimal degree is $3$, so  the initial nonzero term is $x^3$.	In this case, $a_1 = a_2 = 0$, $a_3 = 1$, and $a_4, a_5$ are arbitrary elements from the set $\{0,1\}$, resulting in $|V_3| = 2^2 = 4.$ The constituents of \( V_{3} \) are \[ V_3 = \{ x^3, x^3 + x^4, x^3 + x^5, x^3 + x^4 + x^5 \}. \]
 For $V_4$, the minimal degree is $4$, hence the initial nonzero term is $x^4$.  Here $a_1 = a_2 = a_3 = 0$, $a_4 = 1$, and $a_5 \in \{0,1\}$ arbitrary, so $|V_4| = 2^1 = 2.$ Thus, $V_4 = \{x^4, x^4 + x^5\}.$ For $V_5$, the minimal degree is $5$, so, the only nonzero term is $x^5$. 	Here $a_1 = a_2 = a_3 = a_4 = 0$, $a_5 = 1$,  and $|V_5| = 1, V_5 = \{x^5\}.$ Thus $V( \varGamma(R)) = V_1 \cup V_2 \cup V_3 \cup V_4 \cup V_5,$ and  $|V( \varGamma(R))|=16 + 8 + 4 + 2 + 1 = 31.$
 
 Recall that two vertices $f,g$ in $ \varGamma(R)$ are adjacent if and only if $i+j\ge 6$, where $i = \operatorname{mindeg}(f)$, $j=\operatorname{mindeg}(g)$ with $i,j\in\{1,2,3,4,5\}$. Thus, $V_1$ and $V_2$ are independent sets, and  $V_3$, $V_4$, and $V_5$ are cliques. There are no edges between: 	$V_1 $ and $V_2$, $V_1$ and $V_3, $  $ V_1$ and $V_4, $ and $V_2$ and $V_3.$	Also, $V_1$ is completely joined to $V_5$ (each of the 16 vertices in $V_1$ is adjacent
 to the single vertex $x^5$),  $V_2$ is completely joined to $V_4\cup V_5$ (each of the 8 vertices in $V_2$
 is adjacent to both vertices of $V_4$ and to $x^5$). Also, $V_3\cup V_4\cup V_5$ induces a complete graph $ K_7.$ The graph $ \varGamma(R)$ is shown in Figure \ref{fig 2}.
 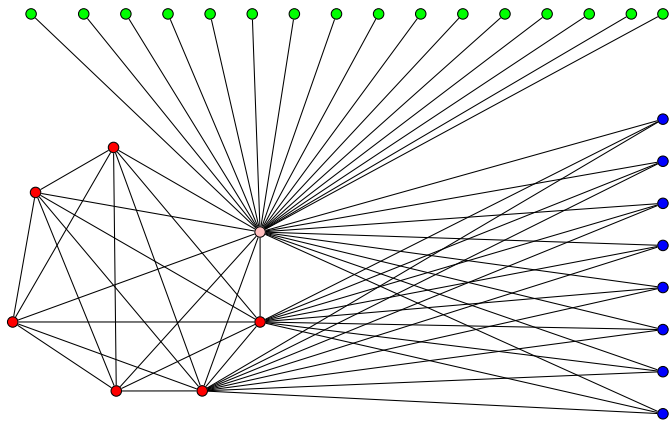
\begin{figure}[H]
 	\centering
 	\begin{tikzpicture}[
 		scale=0.7,
 		vertex/.style={circle, draw, inner sep=1.4pt}, 
 		edge/.style={thin}
 		]
 		
 		\begin{scope}[shift={(-3,0)}]
 			\node[vertex, fill=red] (c1) at (100:2.5)   {};  
 			\node[vertex, fill=red] (c2) at (140:2.5)   {};  
 			\node[vertex, fill=red] (c3) at (200:2.5)   {};  
 			\node[vertex, fill=red] (c4) at (260:2.2)   {};  
 			\node[vertex, fill=red] (c5) at (300:2.5)   {};  
 			\node[vertex, fill=red] (c6) at (340:2.5)   {};  
 			\node[vertex, fill=pink] (c7) at ( 20:2.5)   {};  
 			
 			\foreach \i/\j in {
 				c1/c2, c1/c3, c1/c4, c1/c5, c1/c6, c1/c7,
 				c2/c3, c2/c4, c2/c5, c2/c6, c2/c7,
 				c3/c4, c3/c5, c3/c6, c3/c7,
 				c4/c5, c4/c6, c4/c7,
 				c5/c6, c5/c7,
 				c6/c7}{
 				\draw[edge] (\i) -- (\j);
 			}
 		\end{scope}
 		
 		\node[vertex, fill=blue] (z1) at (7,  3.0) {};
 		\node[vertex, fill=blue] (z2) at (7,  2.2) {};
 		\node[vertex, fill=blue] (z3) at (7,  1.4) {};
 		\node[vertex, fill=blue] (z4) at (7,  0.6) {};
 		\node[vertex, fill=blue] (z5) at (7, -0.2) {};
 		\node[vertex, fill=blue] (z6) at (7, -1.0) {};
 		\node[vertex, fill=blue] (z7) at (7, -1.8) {};
 		\node[vertex, fill=blue] (z8) at (7, -2.6) {};
 		
 		\foreach \z in {z1,z2,z3,z4,z5,z6,z7,z8}{
 			\draw[edge] (\z) -- (c5);
 			\draw[edge] (\z) -- (c6);
 			\draw[edge] (\z) -- (c7);
 		}
 		
 		\node[vertex, fill=green] (x1)  at (-4.0, 5) {};
 		\node[vertex, fill=green] (x2)  at (-3.2, 5) {};
 		\node[vertex, fill=green] (x3)  at (-2.4, 5) {};
 		\node[vertex, fill=green] (x4)  at (-1.6, 5) {};
 		\node[vertex, fill=green] (x5)  at (-0.8, 5) {};
 		\node[vertex, fill=green] (x6)  at ( 0.0, 5) {};
 		\node[vertex, fill=green] (x7)  at ( 0.8, 5) {};
 		\node[vertex, fill=green] (x8)  at ( 1.6, 5) {};
 		\node[vertex, fill=green] (x9)  at ( 2.4, 5) {};
 		\node[vertex, fill=green] (x10) at ( 3.2, 5) {};
 		\node[vertex, fill=green] (x11) at ( 4.0, 5) {};
 		\node[vertex, fill=green] (x12) at ( 4.8, 5) {};
 		\node[vertex, fill=green] (x13) at ( 5.6, 5) {};
 		\node[vertex, fill=green] (x14) at ( 6.4, 5) {};
 		\node[vertex, fill=green] (x15) at ( 7,   5) {};
 		\node[vertex, fill=green] (x16) at (-5.0, 5) {};
 		
 		\foreach \v in {x1,x2,x3,x4,x5,x6,x7,x8,
 			x9,x10,x11,x12,x13,x14,x15,x16}{
 			\draw[edge] (\v) -- (c7);
 		}
 		
 	\end{tikzpicture}
 	\caption{zero-divisor graph $ \varGamma(R)$ for $R \cong \mathbb{F}_2[x]/(x^6)$. }
 	\label{fig 2}
 \end{figure}
 
 The following result gives the $A_{\alpha}$ spectra of $\mathbb{Z}_p[x]/\langle x^{2b}\rangle.$
\begin{theorem}\label{spectra thm even}
	Let $p$ be a prime and $a=2b$ with $b\in \mathbb{N}$, and let $ \varGamma(R)$ be the zero-divisor graph of $R = \mathbb{Z}_p[x]/\langle x^{2b}\rangle \cong \mathbb{F}_p[x]/(x^{2b})$. 	Then the $A_\alpha$--spectrum of $ \varGamma(R)$ is as follows: For each $i=1,\dots,2b-1$, there is an eigenvalue
		\[
		\lambda_i =
		\begin{cases}
			\alpha\bigl(p^{ i}-1\bigr), & \text{if } 1\le i \le b-1,\\[3pt]
			\alpha\bigl(p^{ i}-1\bigr) - 1, & \text{if } b\le i \le 2b-1,
		\end{cases}
		\]
		with multiplicity $n_i - 1 = (p-1)p^{ 2b-1-i} - 1$. The remaining $2b-1$ eigenvalues are the eigenvalues
		$\theta_1,\dots,\theta_{2b-1}$ of the $(2b-1)\times(2b-1)$ matrix
		$B=(B_{ij})_{1\le i,j\le 2b-1}$, where
		\[
		B = \alpha D^\ast + (1-\alpha)Q,\quad \text{with}\quad
		D^\ast = \operatorname{diag}(d_1,\dots,d_{2b-1})
		\]
		and $Q=(Q_{ij})$ the quotient matrix of $A( \varGamma(R))$
		with respect to the partition $V( \varGamma(R))=\bigcup_{i=1}^{2b-1}V_i$, given by
		\[
		Q_{ij} =
		\begin{cases}
			0, & \text{if } i=j,\ 1\le i\le b-1\ (2i<2b),\\[3pt]
			n_i - 1, & \text{if } i=j,\ b\le i\le 2b-1\ (2i\ge 2b),\\[3pt]
			n_j, & \text{if } i\ne j,\ i+j\ge 2b,\\[3pt]
			0, & \text{if } i\ne j,\ i+j<2b.
		\end{cases}
		\]
		Equivalently,
		\[
		B_{ij} =
		\begin{cases}
			\alpha d_i, & \text{if } i=j,\ 1\le i\le b-1,\\[3pt]
			\alpha d_i + (1-\alpha)(n_i-1), & \text{if } i=j,\ b\le i\le 2b-1,\\[3pt]
			(1-\alpha)n_j, & \text{if } i\ne j,\ i+j\ge 2b,\\[3pt]
			0, & \text{if } i\ne j,\ i+j<2b,
		\end{cases}
		\]
		where for $1\le i\le 2b-1$, the degree of any $v\in V_i$ is
		\[
		d_i = \deg(v) =
		\begin{cases}
			p^{ i} - 1, & \text{if } 1\le i \le b-1 \ (2i<2b),\\[3pt]
			p^{ i} - 2, & \text{if } b\le i \le 2b-1 \ (2i\ge 2b).
		\end{cases}
		\]
\end{theorem}

\begin{proof}
By Theorem \ref{even case}, the vertex set of the zero-divisor graph $ \varGamma(R)$ is decomposed into partition of subsets $V_i$ with $n_i = |V_i| = (p-1)p^{ 2b-1-i}, $ and $\sum_{i=1}^{2b-1}n_i = p^{ 2b-1}-1.$	By the adjacency rule $	f\in V_i, g\in V_j, $ with $f\ne g  $ and $f $ is adjacent to $ g  $ if and only if $ i+j\ge 2b.$ Thus,  for $1\le i\le b-1$, then $2i<2b$, so $V_i$ is an independent set, and 		no edges inside $V_i$. A vertex in $V_i$ is adjacent to all vertices in $V_j$ with $j\ge 2b-i$. The degree of $ v\in V_{i}$ is $\deg(v)=p^{ i}-1$ when $1\le i\le b-1$. 	If $b\le i\le 2b-1$, then $2i\ge 2b$, so $V_i$ is a clique, and 
		each vertex in $V_i$ is adjacent to the other $n_i-1$ vertices in $V_i$
		and to all vertices in $V_j$ for $j\ge 2b-i$, $j\ne i$. Thus, for $v\in V_{i}$, the degree of vertex $v$ is		$\deg(v) = p^{ i}-2$, when $b\le i\le 2b-1$. Thus, the degree of every vertex is
		\[
		d_i = \deg(v) =
		\begin{cases}
			p^{ i} - 1, & \text{if } 1\le i \le b-1 \ (2i<2b),\\[3pt]
			p^{ i} - 2, & \text{if } b\le i \le 2b-1 \ (2i\ge 2b).
		\end{cases}
		\]
	Now, for each $i=1,\dots,2b-1$, define
	\[
	W_i = \Bigl\{x\in\mathbb{R}^{V( \varGamma(R))} :
	\mathrm{supp}(x)\subseteq V_i,\ \sum_{v\in V_i} x_v = 0\Bigr\},
	\]
	where  for $x = (x_v)_{v \in V( \varGamma(R))} \in \mathbb{R}^{V( \varGamma(R))},$ we have $\operatorname{supp}(x)= \{  v \in V( \varGamma(R)) : x_v \neq 0  \}.$ 	Then $\dim W_i = |V_i|-1 = n_i -1$. We show that each $W_i$ is invariant under $A_\alpha$ and compute the eigenvalue. Let $x\in W_i$, so $x$ is supported in $V_i$ and $\sum_{v\in V_i}x_v=0$. For $1\le i\le b-1$, and  $v\in V_i$, there are no neighbors of $v$ in $V_i$, and $x$ vanishes on all
	other subsets $V_{j}$, so it implies that $(Ax)_v=0$. For $v\in V_k$ with $k\ne i$, if $i+k<2b$, then $v$ has no neighbors in $V_i$, and hence $(Ax)_v=0$.   If $i+k\ge 2b$, then $v$ is adjacent to all vertices in $V_i$, hence by the defining condition of $W_i$, we have
	\[
	(Ax)_v = \sum_{u\in V_i} x_u = 0.
	\]
	Thus $Ax=0$. On the other hand, $Dx=d_i x=(p^{ i}-1)x$ on $V_i$, and therefore
	\[
	A_\alpha x = \alpha Dx + (1-\alpha)Ax
	= \alpha(p^{ i}-1)x.
	\]
	So every nonzero vector in $W_i$ is an eigenvector with eigenvalue
	\[
	\lambda_i = \alpha(p^{ i}-1),\quad 1\le i\le b-1.
	\]
	For \( b \leq i \leq 2b-1 \), with \( v \in V_i \), the neighbors of \( v \) in the support of \( x \) comprise the other vertices of \( V_i \). Thus, we have \[ (Ax)_v = \sum_{\substack{u \in V_i \\ u \neq v}} x_u = \sum_{u \in V_i} x_u - x_v = -x_v. \] For $v\in V_k$ with $k\ne i$, if $i+k<2b$, then $v$ lacks neighbors in $V_i$, so it gives $(Ax)_v=0$.
	 If $i+k\ge 2b$, then $v$ is adjacent to all vertices in $V_i$, and it gives
	\[
	(Ax)_v = \sum_{u\in V_i} x_u = 0.
	\]
	Hence $Ax=-x$ on its support, and $Dx=d_i x =(p^{ i}-2)x$ on $V_i$. Thus
	\[
	A_\alpha x = \alpha Dx + (1-\alpha)Ax
	= \alpha(p^{ i}-2)x + (1-\alpha)(-x)
	= \bigl(\alpha(p^{ i}-1)-1\bigr)x.
	\]
	So, each nonzero vector in $W_i$ is an eigenvector with eigenvalue $\lambda_i = \alpha(p^{ i}-1)-1,$ for $b\le i\le 2b-1.$	So, $W_i$ is $A_\alpha$--invariant and contributes an eigenvalue $\lambda_i$ with multiplicity $\dim W_i=n_i-1$. Summing over $i=1,\dots,2b-1$, the overall dimension of $W = \bigoplus_{i=1}^{2b-1} W_i$	is
	\[
	\sum_{i=1}^{2b-1}(n_i-1) = \left(\sum_{i=1}^{2b-1}n_i\right) - (2b-1) = \bigl(p^{ 2b-1}-1\bigr) - (2b-1) = p^{ 2b-1} - 2b.
	\]
	Now, take in consideration vectors that are constant on each subset $V_i$ as
	\[
	U_0 = \{x\in\mathbb{R}^{V( \varGamma(R))} : x \text{ is constant on each }V_i\}.
	\]
	Then $\dim U_0 = 2b-1$, as $\mathbb{R}^{V( \varGamma(R))} = W \oplus U_0,$ and 
	\[
	\dim W + \dim U_0 = (p^{ 2b-1}-2b)+(2b-1) = p^{ 2b-1}-1 = |V( \varGamma(R))|.
	\]
	Clearly, the partition $\{V_i\}$ is equitable as every vertex in $V_i$ has the same number
	of neighbors in each $V_j$. Thus $U_0$ is invariant under $A$, and hence under  $A_\alpha$. Let $x\in U_0$ and write $x_i$ for the common value of $x$ on $V_i$. Encode $x$
	by the column vector
	\[
	\mathbf{x} = (x_1,\dots,x_{2b-1})^T \in \mathbb{R}^{2b-1}.
	\]
	For $v\in V_i$, we have
	\[
	(Ax)_v = \sum_{j=1}^{2b-1} Q_{ij} x_j,
	\]
	where $Q_{ij}$ is the number of neighbors in $V_j$ of a given vertex in $V_i$.
	By the adjacency rule $i+j\ge 2b$ in $ \varGamma(R)$.  If $1\le i\le b-1$, then $2i<2b$, so $V_i$ has no internal edges, and $Q_{ii} = 0.$
		For $j\ne i$, a vertex in $V_i$ is adjacent to all vertices of $V_j$ if $j\ge 2b-i$, and to none otherwise. Thus, we have
		\[
		Q_{ij} =
		\begin{cases}
			n_j, & j\ne i,\ j\ge 2b-i,\\[2pt]
			0, & j\ne i,\ j< 2b-i.
		\end{cases}
		\]
		 If $b\le i\le 2b-1$, then $V_i$ is a clique, so each vertex in $V_i$ is adjacent
		to the other $n_i-1$ vertices in $V_i$ and to all vertices in $V_j$ for $j\ge 2b-i$,
		$j\ne i$. Hence, we have
		\[
		Q_{ii} = n_i-1,\qquad
		Q_{ij} =
		\begin{cases}
			n_j, & j\ne i,\ j\ge 2b-i,\\[2pt]
			0, & j\ne i,\ j< 2b-i.
		\end{cases}
		\]
	Now, on $U_0$ it is evident that $Ax  $ is equivalent to $ Q\mathbf{x},$ and $Dx  $ is equivalent to $ D^\ast \mathbf{x},$	where $D^\ast = \operatorname{diag}(d_1,\dots,d_{2b-1})$.
	Therefore, we obtain $A_\alpha x
	= \alpha D x + (1-\alpha) A x$ is equivalent to $\bigl(\alpha D^\ast + (1-\alpha)Q\bigr)\mathbf{x}= B \mathbf{x}.$
	Thus, the restriction of $A_\alpha$ to $U_0$ is represented by $B$
	in the basis $\{\mathbf{1}_{V_1},\dots,\mathbf{1}_{V_{2b-1}}\}$,
	where $\mathbf{1}_{V_i}$ is the indicator of $V_i$. We note that indicator vector $\mathbf{1}_{V_i}$ is a vector which is $1$ for all coordinate $x_{i}$ of eigenvector $X$ on $V_{i},$ and $0$ on other sets $V_{j}$.
	As $\dim U_0 = 2b-1$, so the eigenvalues of this restriction are exactly the simple eigenvalues $\theta_1,\dots,\theta_{2b-1}$ of $B$. So, $B=(B_{ij})_{1\le i,j\le 2b-1}$ in $2\times 2$ can be put as $B =
	\begin{pmatrix}
		B^{(11)} & B^{(12)}\\[4pt]
		B^{(21)} & B^{(22)}
	\end{pmatrix},$	where $B^{(11)}$ is of size $(b-1)\times(b-1)$, $B^{(22)}$ is of size $b\times b$, and
	$B^{(12)},B^{(21)}$ have the corresponding rectangular sizes. The top--left block is diagonal:
	$B^{(11)} =
	\begin{pmatrix}
		\alpha d_1   & 0           & \cdots & 0\\
		0            & \alpha d_2  & \cdots & 0\\
		\vdots       & \vdots      & \ddots & \vdots\\
		0            & 0           & \cdots & \alpha d_{b-1}
	\end{pmatrix}.$	
	The top--right block $B^{(12)}$ (of size $(b-1)\times b$) has the form
	\[
	B^{(12)} =
	\begin{pmatrix}
		0              & 0              & \cdots & 0                  & (1-\alpha)n_{2b-1}\\[2pt]
		0              & 0              & \cdots & (1-\alpha)n_{2b-2} & (1-\alpha)n_{2b-1}\\[2pt]
		\vdots         & \vdots         & \ddots & \vdots             & \vdots\\[2pt]
		0              & (1-\alpha)n_{b+1} & \cdots & (1-\alpha)n_{2b-2} & (1-\alpha)n_{2b-1}
	\end{pmatrix}.
	\]
	The bottom--left block $B^{(21)}$ (of size $b\times(b-1)$) is
	\[
	B^{(21)} =
	\begin{pmatrix}
		0              & 0              & \cdots & 0               & (1-\alpha)n_{b-1}\\[2pt]
		0              & 0              & \cdots & (1-\alpha)n_{b-2} & (1-\alpha)n_{b-1}\\[2pt]
		\vdots         & \vdots         & \ddots & \vdots          & \vdots\\[2pt]
		0              & (1-\alpha)n_{2} & \cdots & (1-\alpha)n_{b-2} & (1-\alpha)n_{b-1}\\[2pt]
		(1-\alpha)n_{1} & (1-\alpha)n_{2} & \cdots & (1-\alpha)n_{b-2} & (1-\alpha)n_{b-1}
	\end{pmatrix},
	\]
	
	Finally, the bottom--right block $B^{(22)}$ (of size $b\times b$) is
	\begin{footnotesize}
		\[
	B^{(22)} =
	\begin{pmatrix}
		\alpha d_{b} + (1-\alpha)(n_{b}-1)   & (1-\alpha)n_{b+1} & \cdots & (1-\alpha)n_{2b-1}\\[2pt]
		(1-\alpha)n_{b} & \alpha d_{b+1} + (1-\alpha)(n_{b+1}-1) & \cdots & (1-\alpha)n_{2b-1}\\[2pt]
		\vdots & \vdots & \ddots & \vdots\\[2pt]
		(1-\alpha)n_{b} & (1-\alpha)n_{b+1} & \cdots &
		\alpha d_{2b-1} + (1-\alpha)(n_{2b-1}-1)
	\end{pmatrix},
	\]
	\end{footnotesize}
	 where, for $1\le i\le 2b-1$,
	\[
	d_i =
	\begin{cases}
		p^{ i}-1, & 1\le i\le b-1,\\[3pt]
		p^{ i}-2, & b\le i\le 2b-1,
	\end{cases}
	\qquad
	n_i = (p-1)p^{ 2b-1-i}.
	\]
\end{proof}

We consider the  example $ \varGamma(\mathbb{F}_2[x]/(x^{2b}))$, and discuss it in detail.
\begin{example}\end{example}
	Let $p=2$ and $b=3$, so that $2b=6$ and $2b-1=5$. Consider the ring $R=\mathbb{Z}_2[x]/(x^6)\cong \mathbb{F}_2[x]/(x^6),$	and its zero--divisor graph $ \varGamma(R)$ with the partition $V( \varGamma(R))=\bigcup_{i=1}^{5}V_i$. From Theorem \ref{spectra thm even}, we have $n_i =2^{5-i}, $ for $ i=1,\dots,5,$ and hence $|V( \varGamma(R))| = \sum_{i=1}^5 n_i = 31 = 2^5 - 1$. The degrees of $ \varGamma(R)$ are $d_1 = 2^1-1 = 1,d_2 = 2^2-1 = 3,d_3 = 2^3-2 = 6,d_4 = 2^4-2 = 14, $ and $d_5 = 2^5-2 = 30.$
	From Theorem~\ref{spectra thm even}, the $A_\alpha$--eigenvalues corresponding
	to the subspaces $W_i$ (supported on $V_i$ with sum zero) are: 
	\[
	\begin{aligned}
		\lambda_1 &= \alpha(2^1-1) = \alpha, 
		&\quad \operatorname{mult}(\lambda_1) &= n_1 - 1 = 16 - 1 = 15,\\
		\lambda_2 &= \alpha(2^2-1) = 3\alpha, 
		&\quad \operatorname{mult}(\lambda_2) &= n_2 - 1 = 8 - 1 = 7,\\
		\lambda_3 &= \alpha(2^3-1) - 1 = 7\alpha - 1,
		&\quad \operatorname{mult}(\lambda_3) &= n_3 - 1 = 4 - 1 = 3,\\
		\lambda_4 &= \alpha(2^4-1) - 1 = 15\alpha - 1,
		&\quad \operatorname{mult}(\lambda_4) &= n_4 - 1 = 2 - 1 = 1,\\
		\lambda_5 &= \alpha(2^5-1) - 1 = 31\alpha - 1,
		&\quad \operatorname{mult}(\lambda_5) &= n_5 - 1 = 1 - 1 = 0,
	\end{aligned}
	\]
	where $ \operatorname{mult}(\lambda_i)$ is the multiplicity of eigenvalue $\lambda_{i}.$ In this way, we obtain  $15+7+3+1 = 26$ eigenvalues. The remaining $5$ eigenvalues come from the restriction of $A_\alpha$ to
	the subspace $U_0$ of vectors that are constant on each $V_i$.
	As in Theorem~\ref{spectra thm even}, this restriction is represented by the
	$(2b-1)\times(2b-1)$ matrix $B=\alpha D^\ast + (1-\alpha)Q,$	where $D^\ast = \operatorname{diag}(d_1,\dots,d_5)$ and $Q$ is the
	quotient matrix of the adjacency matrix with respect to the partition
	$V_1,\dots,V_5$.
	For $p=2$, $b=3$, we have $D^\ast = \operatorname{diag}(1,3,6,14,30)$. 	The entries of $Q$ are given by
	\[
	Q_{ij} =
	\begin{cases}
		0,           & i=j,\ i=1,2,\\[2pt]
		n_i - 1,     & i=j,\ i=3,4,5,\\[2pt]
		n_j,         & i\ne j,\ i+j\ge 6,\\[2pt]
		0,           & i\ne j,\ i+j < 6,
	\end{cases},
	\]
	where $(n_1,n_2,n_3,n_4,n_5) = (16,8,4,2,1)$. Thus, $Q =
	\begin{pmatrix}
		0 & 0 & 0 & 0 & 1\\
		0 & 0 & 0 & 2 & 1\\
		0 & 0 & 3 & 2 & 1\\
		0 & 8 & 4 & 1 & 1\\
		16 & 8 & 4 & 2 & 0
	\end{pmatrix}.$
	Hence $B = \alpha D^\ast + (1-\alpha)Q$	is explicitly given as
	\[
	B =
	\begin{pmatrix}
		\alpha & 0 & 0 & 0 & 1-\alpha\\[3pt]
		0 & 3\alpha & 0 & 2(1-\alpha) & 1-\alpha\\[3pt]
		0 & 0 & 3\alpha+3 & 2(1-\alpha) & 1-\alpha\\[3pt]
		0 & 8(1-\alpha) & 4(1-\alpha) & 13\alpha+1 & 1-\alpha\\[3pt]
		16(1-\alpha) & 8(1-\alpha) & 4(1-\alpha) & 2(1-\alpha) & 30\alpha
	\end{pmatrix}.
	\]
	Let $\theta_1,\dots,\theta_5$ denote the eigenvalues of $B$.  Then, the $A_\alpha$--spectrum of $ \varGamma(R)$  is
	\[
	\sigma\bigl(A_\alpha( \varGamma(R))\bigr)
	=
	\bigl\{
	\underbrace{\alpha,\dots,\alpha}_{15},
	\underbrace{3\alpha,\dots,3\alpha}_{7},
	\underbrace{7\alpha-1,7\alpha-1,7\alpha-1}_{3},
	15\alpha-1,
	\theta_1,\theta_2,\theta_3,\theta_4,\theta_5
	\bigr\},
	\]
	where $\theta_1,\dots,\theta_5$ are the eigenvalues of the $5\times5$
	matrix $B$ above.

%
The following are the immediate consequences of Theorem \ref{spectra thm even}, and gives the adjacency and the singless Laplacian spectra of $ \varGamma(\mathbb{Z}_p[x]/\langle x^{2b}\rangle).$
\begin{corollary}\label{cor:adj-Q-even}
	Let $p$ be a prime and $a=2b$ with $b\in\mathbb{N}$, and let $ G= \varGamma(R)$ be the zero-divisor graph of $R = \mathbb{Z}_p[x]/\langle x^{2b}\rangle \cong \mathbb{F}_p[x]/(x^{2b}).$ From \ref{spectra thm even}, let $V(G)=\bigcup_{i=1}^{2b-1}V_i$ be the partition of $ \varGamma(R)$ with
	$n_i = |V_i| = (p-1)p^{ 2b-1-i}$, and
	\[
	d_i = \deg(v) =
	\begin{cases}
		p^{ i}-1, & 1\le i\le b-1,\\[3pt]
		p^{ i}-2, & b\le i\le 2b-1,
	\end{cases}
	\quad v\in V_i.
	\]
	Let $Q^\ast$ be the quotient matrix of $A(G)$ with respect to this partition, as in Theorem~\ref{spectra thm even}, and let $D^\ast = \operatorname{diag}(d_1,\dots,d_{2b-1}).$
	Then, we have the following.
	\begin{enumerate}
		\item The adjacency spectrum of $G$ is
		\[
		\{0\}^{\left[\sum_{i=1}^{b-1}(n_i-1)\right]}
		\cup
		\{-1\}^{\left[\sum_{i=b}^{2b-1}(n_i-1)\right]}
		\cup
		\{\eta_1,\dots,\eta_{2b-1}\},
		\]
		where $\eta_1,\dots,\eta_{2b-1}$ are the eigenvalues of the quotient matrix $Q^\ast$.
		
		\item The signless Laplacian spectrum of $G$, where $Q(G)=D(G)+A(G)$, is
		\[
		\bigcup_{i=1}^{b-1} \{(p^{ i}-1)^{[ n_i-1 ]}\} \cup \bigcup_{i=b}^{2b-1} \{(p^{ i}-3)^{[ n_i-1 ]}\} \cup 	\{\zeta_1,\dots,\zeta_{2b-1}\},
		\]
		where $\zeta_1,\dots,\zeta_{2b-1}$ are the eigenvalues of the $(2b-1)\times(2b-1)$ matrix $D^\ast + Q^\ast.$
	\end{enumerate}
\end{corollary}
\begin{proof}
	(1) The adjacency matrix is $A(G)=A_0(G)$. With $\alpha=0$ in Theorem~\ref{spectra thm even}, we have for
	$1\le i\le b-1, \lambda_i(0) = 0,$ with multiplicity $n_i-1,$	and for $b\le i\le 2b-1$, $\lambda_i(0) = -1,$ with multiplicity $n_i-1.$ Theorem \ref{spectra thm even}, further states that the remaining $2b-1$ eigenvalues of $A_0(G)$ are precisely the eigenvalues of
	\[
	B(0) = 0\cdot D^\ast + (1-0)Q^\ast = Q^\ast,
	\]
	which is the quotient matrix of $A(G)$ with respect to the partition $\{V_i\}$.\\
	(2) The signless Laplacian of $ \varGamma(R)$ is
	\[
	Q(G) = D(G)+A(G) = 2A_{1/2}(G).
	\]
	If $\mu$ is an eigenvalue of $A_{1/2}(G)$ of multiplicity $m$, then $2\mu$ is an eigenvalue of $Q(G)$ with the same multiplicity $m$. From Theorem~\ref{spectra thm even}, for $\alpha=\tfrac12$ and $1\le i\le b-1,$ we obtain
	\[
	\lambda_i\Bigl(\frac12\Bigr)
	= \tfrac12\bigl(p^{ i}-1\bigr),
	\quad\text{ with multiplicity }n_i-1,
	\]
	and for $b\le i\le 2b-1$,
	\[
	\lambda_i\Bigl(\frac12\Bigr)
	= \tfrac12\bigl(p^{ i}-1\bigr)-1
	= \frac{p^{ i}-3}{2},
	\quad\text{with multiplicity }n_i-1.
	\]
	Thus, the corresponding eigenvalues of $Q(G)=2A_{1/2}(G)$ are
	\[
	p^{ i}-1,\quad 1\le i\le b-1,
	\quad\text{and}\quad
	p^{ i}-3,\quad b\le i\le 2b-1,
	\]
	with multiplicities $n_i-1$.	For the remaining $2b-1$ eigenvalues, we see that
	\[
	B\Bigl(\frac12\Bigr) = \frac12D^\ast + \frac12Q^\ast = \frac12\bigl(D^\ast + Q^\ast\bigr).
	\]
	Hence the remaining $2b-1$ eigenvalues of $Q(G)$ are exactly twice the eigenvalues of
	$B\bigl(\tfrac12\bigr)$, that is, they are the eigenvalues of $D^\ast+Q^\ast$. 
\end{proof}

The following consequence of Theorem \ref{spectra thm even} show that the Laplacian matrix of $ \varGamma(R)\cong  \varGamma(\mathbb{Z}_p[x]/\langle x^{2b}\rangle)$ has only integer eigenvalues.
\begin{corollary}\label{L integral even}
	The Laplacian spectrum of $ \varGamma(R)$ is integral, and its spectrum is 
	\[ \{0\}
	\cup
	\bigcup_{\substack{1\le i\le 2b-1\\ i\neq b}}
	\bigl\{(p^{ i}-1)^{[n_i]}\bigr\}
	\cup
	\bigl\{(p^{ b}-1)^{[n_b-1]}\bigr\}. \]
\end{corollary}
\begin{proof}
	Since, for any $\alpha\neq\beta$,
	\[
	\frac{1}{\alpha-\beta}\bigl(A_\alpha(G)-A_\beta(G)\bigr)
	= D(G)-A(G)=L(G).
	\]
	So, by Theorem \ref{spectra thm even}, we get $$\alpha\bigl(p^{ i}-1\bigr)-\beta\bigl(p^{i}-1\bigr)=(\alpha-\beta)(p^{i}-1).$$
	Thus, for $1\le i \le b-1,$ it follows that $p^{i}-1$  is the Laplacian eigenvalue of $ \varGamma(R)$ with $n_i - 1 = (p-1)p^{ 2b-1-i} - 1$. Similarly, for $b\le i \le 2b-1,$ we see that 
	$$\alpha\bigl(p^{ i}-1\bigr) - 1-\beta\bigl(p^{ i}-1\bigr) + 1=(\alpha-\beta)(p^{i}-1) $$ is the eigenvalue of $ \varGamma(R)$ with multiplicity $n_i - 1 = (p-1)p^{ 2b-1-i} - 1$.  By Theorem~\ref{spectra thm even}, the quotient of $A_\alpha( \varGamma(R))$ 
	with respect to the partition $\{V_i\}_{i=1}^{2b-1}$ is
	\[
	B(\alpha) = \alpha D^\ast + (1-\alpha)Q^\ast,
	\]
	where $D^\ast=\operatorname{diag}(d_1,\dots,d_{2b-1})$ and $Q^\ast$ is the quotient of $A(G)$. Hence for any $\alpha\neq\beta$,
	\[
	\frac{1}{\alpha-\beta}\bigl(B(\alpha)-B(\beta)\bigr)
	= D^\ast - Q^\ast =: \overline{L}.
	\]
	The $(2b-1)\times(2b-1)$ matrix $\overline{L}$ is therefore the quotient of the Laplacian $L(G)$ with respect to
	the partition $\{V_i\}$. With the description of $Q^\ast$ and $d_i$ from Theorem~\ref{spectra thm even}, we have
	\[
	\overline{L}_{ij} =
	\begin{cases}
		d_i, & i=j,\ 1\le i\le b-1,\\
		d_i-(n_i-1)=p^{ i}-1-n_i, & i=j,\ b\le i\le 2b-1,\\
		- n_j, & i\ne j,\ i+j\ge 2b,\\
		0, & i\ne j,\ i+j<2b.
	\end{cases}
	\]
In explicit matrix form, $\overline{L}$ is the $(2b-1)\times(2b-1)$ matrix
\[
\overline{L}=
\begin{pmatrix}
	d_1 & 0   & 0   & \cdots & 0         & 0         & -n_{2b-1} \\[3pt]
	0   & d_2 & 0   & \cdots & 0         & -n_{2b-2} & -n_{2b-1} \\[3pt]
	0   & 0   & d_3 & \cdots & -n_{2b-3} & -n_{2b-2} & -n_{2b-1} \\[3pt]
	\vdots & \vdots & \vdots & \ddots & \vdots    & \vdots    & \vdots \\[3pt]
	0   & 0   & -n_3 & \cdots &
	p^{ 2b-3}-1-n_{2b-3} & -n_{2b-2} & -n_{2b-1} \\[3pt]
	0   & -n_2 & -n_3 & \cdots & -n_{2b-3} &
	p^{ 2b-2}-1-n_{2b-2} & -n_{2b-1} \\[3pt]
	-n_1 & -n_2 & -n_3 & \cdots & -n_{2b-3} & -n_{2b-2} &
	p^{ 2b-1}-1-n_{2b-1}
\end{pmatrix}.
\]
Now, we show that see that $\{0\}$ along with $p^{i}-1$ are the eigenvalue of $\overline{L}$ for $1\leq i\leq 2b-1$ except $i=b.$

	We let $N=2b-1$ and recall that $n_j=(p-1)p^{N-j}$. So, with 	$1\le a\le b\le N$, w ehave
	\begin{align}
		\sum_{j=a}^{b} n_j
		&= (p-1)\sum_{j=a}^{b} p^{N-j}
		= (p-1)\sum_{t=N-b}^{N-a} p^t \notag\\
		&= (p-1) \frac{p^{N-a+1}-p^{N-b}}{p-1}
		= p^{N-a+1}-p^{N-b}. \label{eq:general-sum}
	\end{align}
	In particular, for $r\ge1$, we have
	\begin{equation}\label{eq:tail-sum}
		\sum_{j=N-r+1}^{N} n_j = p^r-1.
	\end{equation}
	Due to positive semidefinite property of the Laplacian of $ \varGamma(R)$, $0$ is the simple eigenvalue of $\overline{L}$. Next, for eigenvalues $p^k-1$ with $1\le k\le b-1$, we fix an integer $k$,  $1\le k\le b-1$ and let $
	\lambda_k = p^k-1.$ Now, consider a vector  $v^{(k)}\in\mathbb{R}^N$ with 
	\begin{equation}\label{eq:v-k}
		v^{(k)}_j =
		\begin{cases}
			0, & 1\le j<k,\\
			- \varGamma_k, & j=k,\\
			1, & k<j\le N-k,\\
			0, & N-k<j\le N,
		\end{cases}
	\end{equation}
	where $ \varGamma_k$ is a scaler, we are to be chosen so that eigenequation is satisfied. For the computation of $(\overline{L}v^{(k)})_i$, we consider the following cases. For  $1\le i<k$, we have $v_i^{(k)}=0$. If $j$ is such that $v_j^{(k)}\ne0$, then 
	$k\le j\le N-k$, so it gives
	\[
	i+j\le (k-1)+(N-k) = N-1 < N+1.
	\]
	Hence $i+j\ge N+1$ never holds, so all off-diagonal entries in row $i$ of $\overline{L}$
	multiply zeros of $v^{(k)}$. Thus, we have
	\[
	(\overline{L}v^{(k)})_i = d_i v^{(k)}_i = 0 = \lambda_k v^{(k)}_i.
	\]
	For $i=k$, we have $v^{(k)}_k=- \varGamma_k$. For $j$ in the support of $v^{(k)}$ with
	$j\ne k$, we have $k<j\le N-k$, so it gives
	\[
	k+j \le k+(N-k) = N < N+1.
	\]
	Thus $k+j\ge N+1$ never holds, and there are no nonzero off-diagonal
	contributions in row $k$. So, $(\overline{L}v^{(k)})_k = d_k v^{(k)}_k.$
	Since $1\le k\le b-1$, we have $d_k=p^k-1=\lambda_k$, and we obtain 
	\[
	(\overline{L}v^{(k)})_k = (p^k-1)v^{(k)}_k
	= \lambda_k v^{(k)}_k,
	\]
	which is independent of $ \varGamma_k$. If $k<i\le N-k$, then in this range $v^{(k)}_i=1$. The only indices $j$ with $v_j^{(k)}\ne 0$	are $j=k$ and $k<j\le N-k$. However, the column $k$ does not appear in the off-diagonal part of row $i$. 	Indeed, for $j=k$, we obtain
	\[
	i+k \le (N-k)+k = N < N+1.
	\]
	So, $i+k\ge N+1$ fails and it gives $(\overline{L})_{ik}=0$. Thus $ \varGamma_k$ plays
	no role in this range. The nonzero off-diagonal terms come from those
	$j$ with $k<j\le N-k, $ with $ i+j\ge N+1.$
	The inequality $i+j\ge N+1$ is $j\ge N+1-i$, so the relevant $j$ are
	\[
	j\in\bigl[\max(k+1,N+1-i), N-k\bigr].
	\]
	Since $i\le N-k$, we get $N+1-i \ge N+1-(N-k)=k+1,$ 	so the lower bound is $j\ge N+1-i$. In summary, the contributing
	indices are
	\[
	j = N+1-i,N+2-i,\dots,N-k.
	\]
	These all satisfy $j>k$, hence $v^{(k)}_j=1$. 	If $i<b$ then $N+1-i > i$, so $j=i$ is not in this interval and the
	sum includes no diagonal index. If $i\ge b$ then $2i\ge 2b=N+1$, so
	$N+1-i \le i$, and $j=i$ \emph{is} in the interval. In this case we
	must omit $j=i$ from the off-diagonal sum. Both these cases can be treated
	uniformly by writing
	\[
	(\overline{L}v^{(k)})_i
	= d_i - \sum_{j=N+1-i}^{N-k} n_j + \varepsilon_i n_i,
	\]
	where $\varepsilon_i=0$ if $i<b$ and $\varepsilon_i=1$ if $i\ge b$. By \eqref{eq:general-sum} (with $a=N+1-i$, $b=N-k$),
	\[
	\sum_{j=N+1-i}^{N-k} n_j
	= p^{N-(N+1-i)+1}-p^{N-(N-k)}
	= p^i-p^k.
	\]
	
	If $k<i\le b-1$, then $d_i=p^i-1$ and $\varepsilon_i=0$.
	Thus, we have
	\[
	(\overline{L}v^{(k)})_i
	= (p^i-1) - (p^i-p^k)
	= p^k-1
	= \lambda_k v^{(k)}_i.
	\]
	If $b\le i\le N-k$, then $d_i=p^i-1-n_i$ and
	$\varepsilon_i=1$, so we obtain 
	\[
	(\overline{L}v^{(k)})_i
	= (p^i-1-n_i) - (p^i-p^k) + n_i
	= p^k-1
	= \lambda_k v^{(k)}_i.
	\]
	Thus the eigenvalue equation holds for all $k<i\le N-k$ for every
	choice of $ \varGamma_k$. With $N-k < i\le N$, here $v^{(k)}_i=0$, so we need $(\overline{L}v^{(k)})_i=0$.
	Now $v^{(k)}_j\ne 0$ only for $k\le j\le N-k$. Since
	$i\ge N-k+1$, for any such $j$ we have
	\[
	i+j \ge (N-k+1)+k = N+1.
	\]
	So, all of these $j$ appear in the off-diagonal sum. Therefore, we have
	\[
	(\overline{L}v^{(k)})_i
	= -\sum_{j=k}^{N-k} n_j v^{(k)}_j
	= -\Bigl(n_k(- \varGamma_k) + \sum_{j=k+1}^{N-k} n_j\cdot1\Bigr)
	=  \varGamma_k n_k - \sum_{j=k+1}^{N-k} n_j.
	\]
	This expression is independent of $i$ in this range. Thus all the
	bottom rows impose the same scalar condition $	 \varGamma_k n_k = \sum_{j=k+1}^{N-k} n_j.$
	In order to  solve  $ \varGamma_k$ explicitly , we use \eqref{eq:general-sum},
	\[
	\sum_{j=k+1}^{N-k} n_j
	= \sum_{j=k}^{N-k} n_j - n_k
	= \bigl(p^{N-k+1}-p^{N-(N-k)}\bigr) - n_k
	= (p^{N-k+1}-p^k) - n_k.
	\]
	But $n_k=(p-1)p^{N-k}$, so $	p^{N-k+1}-p^k-n_k = p^{N-k}-p^k.$
	Hence, with these facts, we obtain $	\sum_{j=k+1}^{N-k} n_j = p^{N-k}-p^k.$ Also, $n_k = (p-1)p^{N-k}$, so we have
	\[
	 \varGamma_k
	= \frac{\sum_{j=k+1}^{N-k} n_j}{n_k}
	= \frac{p^{N-k}-p^k}{(p-1)p^{N-k}}
	= \frac{1-p^{2k-N}}{p-1}
	= \frac{1-p^{ 2k-2b+1}}{p-1}.
	\]
	With the above value of $ \varGamma_k$, all bottom rows give $(\overline{L}v^{(k)})_i=0=\lambda_k v^{(k)}_i$. Therefore, we have 
	\[
	\overline{L}v^{(k)} = \lambda_k v^{(k)},\qquad
	\lambda_k=p^k-1,\quad 1\le k\le b-1.
	\]
	
	\medskip
	
	Now, for the eigenvalues $p^k-1$ of $\overline{L}$ for $b+1\le k\le N$, we fix an integer $k$ with $b+1\le k\le N$ and set again	$\lambda_k=p^k-1$. For $m = N+1-k = 2b-k,$ with $1\le m\le b-1<k$, we define $w^{(k)}\in\mathbb{R}^N$ by
	\begin{equation}\label{eq:w-k}
		w^{(k)}_j =
		\begin{cases}
			0,          & 1\le j<m,\\
			-\delta_k,  & m\le j\le k-1,\\
			1,          & j=k,\\
			0,          & k<j\le N,
		\end{cases}
	\end{equation}
	where $\delta_k$ is the  scalar, which we are required to find that $w^{(k)}$ is eigenvector of some eigenvalue of $\overline{L}$. Again we compute $(\overline{L}w^{(k)})_i$ by different ranges over $i$. For $1\le i<m$, we have $w^{(k)}_i=0$. For any $j$ with $w^{(k)}_j\ne0$, we have$m\le j\le k$, and thus
	\[
	i+j \le (m-1)+k = (N+1-k-1)+k = N.
	\]
	Hence $i+j\ge N+1$ never holds, so the off-diagonal sum of $\overline{L}$ is empty,  and we obtain
	\[
	(\overline{L}w^{(k)})_i = d_i w^{(k)}_i = 0 = \lambda_k w^{(k)}_i.
	\]
	For $i=k$, we have $w^{(k)}_k=1$. The only nonzero $w^{(k)}_j$ with $j\ne k$ occur
	for $m\le j\le k-1$, and for all such $j$,  we have
	\[
	k+j \ge k+m = k+(N+1-k)=N+1,
	\]
	and they all contribute to the off-diagonal sum. Thus, we obtain
	\[
	(\overline{L}w^{(k)})_k
	= d_k\cdot 1 - \sum_{j=m}^{k-1} n_j(-\delta_k)
	= d_k + \delta_k\sum_{j=m}^{k-1} n_j.
	\]
	As $k\ge b+1$, so $k\ge b$ and hence $d_k=p^k-1-n_k$. In order for $(\overline{L}w^{(k)})_k=\lambda_k w^{(k)}_k=p^k-1$ to hold, so we must have
	\[
	p^k-1-n_k + \delta_k\sum_{j=m}^{k-1}n_j = p^k-1,
	\]
	which is equivalent to $\delta_k\sum_{j=m}^{k-1}n_j = n_k.$
	Using \eqref{eq:general-sum} with $a=m$, $b=k-1$, we obtain
	\[
	\sum_{j=m}^{k-1} n_j
	= p^{N-m+1}-p^{N-(k-1)}
	= p^{k}-p^{N-k+1}.
	\]
	Thus
	\[
	\delta_k
	= \frac{n_k}{p^k-p^{N-k+1}}
	= \frac{(p-1)p^{N-k}}{p^k-p^{N-k+1}}= \frac{p-1}{p^{2k-N}-p}
	= \frac{p-1}{p\bigl(p^{2(k-b)}-1\bigr)}.
	\]
	With the choice of $\delta_k$, for $m\le i\le k-1$, we have  $w^{(k)}_i=-\delta_k$. The nonzero $w^{(k)}_j$ occur for $j\in[m,k-1]$ (where $w_j=-\delta_k$) and $j=k$ (where $w_k=1$). With	$m\le i\le k-1$, we have $i+k \ge m+k = N+1,$	so $j=k$ always appears in the off-diagonal sum. Among $m\le j\le k-1$, the condition $i+j\ge N+1$ is $j\ge N+1-i$. Since
	\[
	N+1-i \ge N+1-(k-1) = N-k+2 = m+1,
	\]
	we have $N+1-i \ge m$ and the contributing $j$ satisfy $j\in [N+1-i,k-1].$ 	We must again check whether $i$ lies in this interval. If $m\le i<b$, then  $N+1-i > i$, and it shows that  $i$ is not in $[N+1-i,k-1]$. In this case
	\[
	(\overline{L}w^{(k)})_i
	= d_i(-\delta_k) - \sum_{j=N+1-i}^{k-1} n_j(-\delta_k) - n_k\cdot 1
	= -\delta_k d_i + \delta_k\sum_{j=N+1-i}^{k-1} n_j - n_k.
	\]
	Using \eqref{eq:general-sum} with $a=N+1-i$, $b=k-1$, we have
	\[
	\sum_{j=N+1-i}^{k-1} n_j
	= p^{N-(N+1-i)+1}-p^{N-(k-1)}
	= p^{i}-p^{N-k+1}.
	\]
	Since $i<b$, $d_i=p^i-1$, so
	\begin{align*}
		(\overline{L}w^{(k)})_i
		&= -\delta_k(p^i-1) + \delta_k(p^i-p^{N-k+1}) - n_k= -\delta_k(p^{N-k+1}-1) - n_k.
	\end{align*}
	On the other hand, we have
	\[
	\lambda_k w^{(k)}_i = (p^k-1)(-\delta_k) = -\delta_k(p^k-1).
	\]
	Thus the eigenvalue equation is equivalent to
	\[
	-\delta_k(p^{N-k+1}-1) - n_k = -\delta_k(p^k-1)
	\]
	or
	\[
	\delta_k\bigl(p^k-p^{N-k+1}\bigr) = n_k.
	\]
	But this is precisely how $\delta_k$ was chosen above. Hence
	$(\overline{L}w^{(k)})_i=\lambda_k w^{(k)}_i$ for $m\le i<b$. Now, for $b\le i\le k-1$, we get $2i\ge 2b=N+1$. So, $N+1-i\le i$ and the interval
	$[N+1-i,k-1]$ contains $i$. Thus, we have
	\[
	\sum_{\substack{m\le j\le k-1\\ i+j\ge N+1}} n_j(-\delta_k)
	= -\delta_k\left(\sum_{j=N+1-i}^{k-1} n_j - n_i\right),
	\]
	and
	\[
	(\overline{L}w^{(k)})_i
	= d_i(-\delta_k) -\Bigl(-\delta_k\bigl(\sum_{j=N+1-i}^{k-1} n_j - n_i\bigr)\Bigr)
	- n_k
	= -\delta_k d_i + \delta_k\sum_{j=N+1-i}^{k-1}n_j - \delta_k n_i - n_k.
	\]
	Again with $\sum_{j=N+1-i}^{k-1}n_j=p^i-p^{N-k+1}$ and
	$d_i=p^i-1-n_i$, we obtain
	\begin{align*}
		(\overline{L}w^{(k)})_i
		&= -\delta_k(p^i-1-n_i) + \delta_k(p^i-p^{N-k+1}) - \delta_k n_i - n_k= \delta_k - \delta_k p^{N-k+1} - n_k\\
		&= -\delta_k(p^{N-k+1}-1) - n_k.
	\end{align*}
	As in the previous case this equals $-\delta_k(p^k-1)$ exactly when
	$\delta_k(p^k-p^{N-k+1})=n_k$, that is, for our $\delta_k$, and hence for $b\le i\le k-1$, we obtain
	\[
	(\overline{L}w^{(k)})_i = \lambda_k w^{(k)}_i.
	\]
	For the last range $i>k$, we have $w^{(k)}_i=0$. For such $i$ and with any $j\in[m,k]$, we have
	\[
	i+j \ge (k+1)+m = (k+1)+(N+1-k)=N+2>N+1.
	\]
	So, all indices in the support $\{m,\dots,k\}$ contribute to the off-diagonal sum. Thus, with the defining relation for $\delta_k$, we obtain
	\begin{align*}
		(\overline{L}w^{(k)})_i
		&= -\sum_{j=m}^{k-1} n_j(-\delta_k) - n_k\cdot1\\
		&= \delta_k\sum_{j=m}^{k-1}n_j - n_k
		= \delta_k(p^k-p^{N-k+1}) - n_k
		= 0.
	\end{align*}
	Since $\lambda_k w^{(k)}_i=0$
	as well, the eigenvalue equation holds. Thus,  it shows that for $b+1\le k\le N$,
	\[
	\overline{L}w^{(k)} = (p^k-1)w^{(k)}.
	\]
	The numbers $p^k-1$ are strictly increasing in $k$, so the
	$2b-2$ values $\{p^k-1:1\le k\le N,\ k\ne b\}$ are all distinct, and
	together with $0$ they give $2b-1=N$ distinct eigenvalues. Thus, for every $i$ with $1\le i\le 2b-1$ and $i\neq b$,
	$p^i-1$ is an eigenvalue of $\overline{L}$.
\end{proof}

\section{Spectra of zero-divisor graph of $\mathbb{Z}_p[x]/\langle x^{2b+1}\rangle $}\label{section 3}
In this section, we consider the $A_{\alpha}$ spectra of the zero-divisor graph of $\mathbb{Z}_p[x]/\langle x^{2b+1}\rangle $. First, we understand the combinatorial structure of $  \varGamma(\mathbb{Z}_p[x]/\langle x^{2b+1}\rangle).$
\begin{theorem} \label{odd case}
	Let $p$ be a prime and let $b\ge 1$ be an integer, and let $ \varGamma(R)$ be the zero-divisor graph of $R=\mathbb{Z}_p[x]/\langle x^{2b+1}\rangle \cong \mathbb{F}_p[x]/(x^{2b+1})$. Then the following hold
	\begin{enumerate}
		\item The vertex set $V( \varGamma(R))$ can be decomposed into subsets $V_i
		= \left\{
		\sum_{k=i}^{2b} a_k x^k : a_i \neq 0
		\right\},  1\le i \le 2b$ with $ 	|V_i|
		= (p-1)p^{ 2b-i},  i=1,\dots,2b,$ and each vertex of $V_{i}$ is adjacent to every vertex of $V_{j}$ if and only if $i+j\geq 2b+1.$
		\item The induced subgraph of $V_{i}$ is $\overline{K}_{|V_{i}|}$ if $1\leq i\leq b$ and it is $K_{|V_{i}|}$ if $b+1\leq i\leq 2b.$
		\item The clique number of $ \varGamma(R)$ is $\omega( \varGamma(R))=p^{b}$, the independence number is $p^{2b} - p^b$, and the domination number is $ \varGamma( \varGamma(R))=1$.
		\item The diameter of $ \varGamma(R)$ is $2$ and its girth is $3.$ 
	\end{enumerate}
\end{theorem}
\begin{proof}
We use the description of $R$ as an $\mathbb{F}_p$--vector space, and the characterize the
		units by the constant term in the quotient rings $\mathbb{F}_p[x]/(x^{2b+1})$.	Since $R$ is finite, every nonunit is a zero-divisor, so the nonzero zero-divisors are precisely the nonzero elements of the ideal $(x)$. Thus, as an $\mathbb{F}_p$--vector space, and with multiplication modulo $x^{2b+1}$, we have
		\[
		R = \Bigl\{\sum_{k=0}^{2b} a_k x^k : a_k \in \mathbb{F}_p\Bigr\}.
		\]
		 So, it follows that all terms of degree greater than or equal to $2b+1$ vanish.	An element $\sum_{k=0}^{2b} a_k x^k$ of $R$ is a unit if and only if $a_0\neq 0$. Thus, the set of nonzero zero-divisors of $R$ is
		\[
		Z(R)\setminus\{0\}
		= \Bigl\{\sum_{k=1}^{2b} a_k x^k : (a_1,\dots,a_{2b}) \neq (0,\dots,0)\Bigr\},
		\]
		which are precisely the nonzero elements of the ideal $(x)$. So, it follows that $|V( \varGamma(R))|
		= |Z(R)\setminus\{0\}|
		= p^{ 2b} - 1.$ For a nonzero zero-divisor $f \in Z(R)$,  the minimal degree is defined as
		\[
		\operatorname{mindeg}(f)
		= \min\bigl\{k \in \{1,\dots,2b\} : \text{the coefficient of $x^k$ in $f$ is nonzero}\bigr\}.
		\]
		This partitions the vertex set $V( \varGamma(R)) = Z(R)\setminus\{0\}$ into 
		$$V( \varGamma(R)) = V_1 \cup V_2 \cup \dots V_{b-1}\cup V_{b}\cup V_{b+1}\cup \dots \cup V_{2b},$$
		where
		\[
		V_i= \left\{\sum_{k=i}^{2b} a_k x^k : a_i \neq 0
		\right\}, \qquad 1\le i \le 2b.
		\]
		For each $i$, the coefficient $a_i$ can be any nonzero element of $\mathbb{F}_p$,
		and the coefficients $a_{i+1},\dots,a_{2b}$ are arbitrary. Hence $	|V_i|
		= (p-1)p^{ 2b-i}, $ for $ i=1,\dots,2b,$		and we have
		\[
		\sum_{i=1}^{2b} |V_i|
		= (p-1)\sum_{t=0}^{2b-1} p^t
		= p^{ 2b}-1
		= |V( \varGamma(R))|.
		\]
		Let $f,g\in Z(R)\setminus\{0\}$ with $\operatorname{mindeg}(f)=i$ and
		$\operatorname{mindeg}(g)=j$, where $1\le i,j\le 2b$. Write
		\[
		f = a_i x^i + \text{(higher degree terms)}, \qquad
		g = b_j x^j + \text{(higher degree terms)},
		\]
		with $a_i,b_j \in \mathbb{F}_p^\times$. In  $fg$, the term of minimal degree is $x^{i+j}$, whose coefficient is $a_i b_j$,
		since all other products of coefficients have strictly larger degree.
		Thus, it implies that	$fg = 0  $ if and only if $ i+j \ge 2b+1.$	Clearly, if $i+j \le 2b$, then the $x^{i+j}$ term is nonzero in $R$, and  if $i+j\ge 2b+1$,	then every term of $fg$ has degree $\ge 2b+1$ and vanishes modulo $(x^{2b+1})$. Consequently, the adjacency between two vertices in $ \varGamma(R)$ is completely determined by the minimal degrees, that is, $f \in V_i,\ g\in V_j,\ f\neq g$ then $f\sim g$ if and only if $i+j \ge 2b+1.$ In particular, for the induced subgraphs on the sets $V_i$, two distinct vertices in $V_i$ are adjacent if and only if $2i\ge 2b+1$. Since $2i$ is an integer, this is equivalent to $i\ge b+1$. Thus $V_i$  is an independent set if $1\le i \le b$, and $V_{i}$ is a clique  if  $b+1\le i \le 2b.$ For $i\neq j$, either all vertices of $V_i$ are adjacent to every vertex of $V_j$, or there exist no edges between them. In particular, no vertex of $V_i $ is adjacent to any $V_j,$ if $i+j \le 2b$, and each vertex of $V_{i}$ is adjacent to every vertex of $V_{j}$ if $i+j \ge 2b+1.$
		
		A particularly important subgraph is given by the subset $C = \bigcup_{k=b+1}^{2b} V_k.$ 
		So, for all $i,j\ge b+1$, it gives us $ i+j \ge (b+1)+(b+1) = 2b+2 > 2b+1.$ Thus, every pair of vertices in $C$ is adjacent, and  $C$ induces a clique of size
		\[
		|C|
		= \sum_{k=b+1}^{2b} |V_k|
		= \sum_{k=b+1}^{2b} (p-1)p^{ 2b-k}
		= (p-1)\sum_{t=0}^{b-1} p^t
		= p^{ b} - 1.
		\]
		Moreover, every vertex in $V_b$ is adjacent to every vertex in $C$, since $b + k \ge b+(b+1)=2b+1 $ for all $k\ge b+1.$	Thus, any set of the form $\{v\}\cup C$ with $v\in V_b$ is a clique of size $	|C|+1 = (p^b - 1)+1 = p^b.$ To identify the clique number, from above we note that  $\omega( \varGamma(R))\ge p^b$.	On the other hand, let $S$ be any clique and let $i_0=\min\{\operatorname{mindeg}(f): f\in S\}$.	If $i_0\ge b+1$, then $S\subseteq C$, so $|S|\le |C|=p^b-1$.
		If $i_0\le b$, then for any $g\in S$ with $\operatorname{mindeg}(g)=j$, we must have
		$i_0+j\ge 2b+1$, so $j\ge 2b+1-i_0\ge b+1$, and hence $S\subseteq \{f\}\cup C$ for some
		$f\in V_{i_0}$. Thus $|S|\le 1+|C|=p^b$, and it follows that $\omega( \varGamma(R))=p^b$.
		For $1\le i\le b$, each $V_i$ is independent,and if $1\le i,j\le b$ then $i+j \le b+b = 2b < 2b+1,$
		so there are no edges between $V_i$ and $V_j$ either. Thus, $S_0= \bigcup_{i=1}^b V_i$ is an independent set with cardinality 
		\[
		|S_0|
		= \sum_{i=1}^b |V_i|
		= \sum_{i=1}^b (p-1)p^{2b-i}
		= (p-1)\sum_{t=b}^{2b-1} p^t
		= (p-1)\frac{p^{2b}-p^b}{p-1}
		= p^{2b} - p^b.
		\]
		So, it follows that	$\alpha( \varGamma(R))\ge p^{2b} - p^b.$  We now show that no independent set can be larger than $S_0$. Let $S$ be an arbitrary independent set in $ \varGamma(R)$. If $S \cap \bigcup_{i=b+1}^{2b} V_i = \varnothing$, then $S \subseteq \bigcup_{i=1}^b V_i = S_0$, so $|S| \le |S_0|$. If $S$ contains a vertex from some $V_j$ with $j\ge b+1$.	Since $V_j$ is a clique for $j\ge b+1$, $S$ can contain at most one vertex from each such $V_j$,	and in particular, at most one from the entire set $\bigcup_{i=b+1}^{2b}V_i$. Let $	j_0 = \max\{j : S\cap V_j \neq \varnothing\},  $ with $ j_0\ge b+1.$ Then $S$ contains exactly one vertex from $V_{j_0}$, and no vertices from $V_j$ with $j>j_0$. Now, if $v\in V_{j_0}$, then for a vertex $u\in V_i$ with $i+j_0\ge 2b+1$ we have $u\sim v$,	so $u$ cannot belong to $S$. Thus, $S$ can contain vertices from $V_i$ only when $i + j_0 \le 2b  $, if $ i \le 2b - j_0.$	Consequently, $
		S \subseteq \left(\bigcup_{i=1}^{2b-j_0} V_i\right) \cup (S\cap V_{j_0}),$ 
		and hence	$|S| \le \sum_{i=1}^{2b-j_0} |V_i| + 1.$ 
		Comparing it with $|S_0|$, we have
		\[
		|S_0| - |S|\ge\sum_{i=1}^b |V_i| - \left(\sum_{i=1}^{2b-j_0} |V_i| + 1\right)
		=
		\sum_{i=2b-j_0+1}^{b} |V_i| - 1.
		\]
		Since $j_0\ge b+1$, we have $2b-j_0+1 \le b$, so the index set
		$\{2b-j_0+1,\dots,b\}$ is nonempty. Moreover, for each $i$, $|V_i| = (p-1)p^{2b-i} \ge p-1 \ge 1,$ so it gives $\sum_{i=2b-j_0+1}^{b} |V_i| \ge 1,$	and hence $|S_0|-|S|\ge 0$, or $|S|\le |S_0|$. Combining the two cases, every independent set $S$ satisfies $|S|\le |S_0|$, and we already	have an independent set $S_0$ of size $p^{2b}-p^b$. Therefore, we obtain $\alpha( \varGamma(R)) = |S_0| = p^{2b} - p^b.$
		
		As for $i\ge b+1$ the induced subgraph on $V_i$ is a clique,	and for $i=2b, \deg(v) = p^{ 2b} - 2,  $ for $ v\in V_{2b},$	while $|V( \varGamma(R))| = p^{ 2b} - 1$. Thus, each vertex $v\in V_{2b}$ is adjacent to every
		other vertex of $ \varGamma(R)$, that is, every vertex in $V_{2b}$ is a universal vertex. 

		For $v\in V_i$, its neighbors are exactly the vertices in $\bigcup_{j: i+j\ge 2b+1} V_j$,
		with the additional edges inside $V_{i}$, if $2i\ge 2b+1$ or  $i\ge b+1$, then in this case $v$ is also adjacent
		to all other vertices in $V_i$. With $\sum_{j=m}^{2b} |V_j| = p^{ 2b-m+1}-1 $,  for $1\le m\le 2b,$	and taking $m=2b+1-i$, we get the total number of vertices in subsets $V_{2b+1-i},\dots,V_{2b}$ as $\sum_{j=2b+1-i}^{2b} |V_j|= p^{ i}-1.$
		This includes $V_i$ itself exactly when $2i\ge 2b+1$.	Therefore, the degree of $v\in V_{i}$ is $\deg(v)=p^{ i}-1$ if  $1\le i\le b$, and $\deg(v)p^{ i}-2,$  if $b+1\le i\le 2b$. In particular, for $i=2b$, we have $2(2b)\ge 2b+1$, so each vertex in
		$V_{2b}$ has degree $\deg(v) = p^{ 2b} - 2, v\in V_{2b},$ and is adjacent to every other vertex of $ \varGamma(R)$.
		Thus $V_{2b}$ consists of $p-1$ universal vertices. As $V_{2b}$ is nonempty with $p-1$ vertices, and every vertex is adjacent
			to every vertex of $V_{2b}$, the graph $ \varGamma(R)$ is connected for all $b\ge 1$. 
			Moreover, if $b=1$, so $a=3$, then there are two subsets $V_1$ an independent set, $V_2$ a clique, and every vertex in $V_1$ is adjacent to every vertex in $V_2$. Thus, the diameter of $ \varGamma(R)$	is $2$. If $b\ge 2$, so $a\ge 5$, there are nonadjacent pairs in $V_1$, but any two vertices are at distance at most $2$ via a vertex in $V_{2b}$. Hence, for all $b\ge 1$, the diameter of $ \varGamma(R)$ is $2$. For $b\ge 2$, the sets $V_{2b-1}$ and $V_{2b}$ satisfy $2(2b-1)\ge 2b+1,\quad (2b-1)+(2b)\ge 2b+1,$		so $V_{2b-1}$ is a clique and every vertex in $V_{2b-1}$ is adjacent to every vertex in $V_{2b}$. Since $|V_{2b-1}|\ge 2$ and $|V_{2b}|\ge 1$, this yields triangles, so for all primes $p$ and all $b\ge 2$, the girth of $ \varGamma(R)$ is $3$. Here, we ignore the case $a=3$ and $p=2$,  as $ \varGamma(R)$ has no cycles (girth infinite).
\end{proof}

We will illustrate Theorem \ref{odd case} for $b=2$, so that $a=5.$\\
	Let $p$ be a prime and let $R= \mathbb{Z}_p[x]/\langle x^5\rangle \cong \mathbb{F}_p[x]/(x^5).$	Let $ \varGamma(R)$ be its zero--divisor graph. Then  the vertices of $ \varGamma(R)$ are partitioned as $V( \varGamma(R)) = V_1 \cup V_2 \cup V_3 \cup V_4,$
	where
	\[
	\begin{aligned}
		V_1 &= \{a_1x + a_2x^2 + a_3x^3 + a_4x^4 : a_1 \neq 0\},
		& |V_1| &= p^4 - p^3,\\[2pt]
		V_2 &= \{a_2x^2 + a_3x^3 + a_4x^4 : a_2 \neq 0\},
		& |V_2| &= p^3 - p^2,\\[2pt]
		V_3 &= \{a_3x^3 + a_4x^4 : a_3 \neq 0\},
		& |V_3| &= p^2 - p,\\[2pt]
		V_4 &= \{a_4x^4 : a_4 \neq 0\},
		& |V_4| &= p-1.
	\end{aligned}
	\]
	Then, the order of $ \varGamma(R)$  is $p^4 - 1.$ The degree of a vertex $v\in V( \varGamma(R))$ is $
			\deg(v) = p-1, p^{2}-1, p^{3}-2, $ and $p^{4}-2 $ for $ v\in V_1,, v\in V_2, v\in V_3, $ and $ v\in V_4$, respectively. The set $V_3 \cup V_4$ induces a clique of size $|V_3|+|V_4| = p^2 - 1$, and for any
			vertex $v\in V_2$, the set $\{v\} \cup V_3 \cup V_4 $ is a clique of size $p^2$.  The sets $V_1$ and $V_2$ are independent, with no edges joining $V_1$ to $V_2$. Therefore, $V_1\cup V_2$ constitutes an independent set of size $\alpha\bigl( \varGamma(R)\bigr) = |V_1| + |V_2| = p^4 - p$.
	
For $p=2,$   every element of $R=\mathbb{F}_2[x]/(x^5)$ can be written uniquely as $a_0 + a_1x + a_2x^2 + a_3x^3 + a_4x^4,$ where $a_i\in\mathbb{F}_2=\{0,1\},$ with multiplication modulo $x^5$. The units are those with $a_0\neq 0$, so the nonzero zero-divisors are exactly the nonzero
multiples of $x$:
\[
Z(R)\setminus\{0\} = \{a_1x + a_2x^2 + a_3x^3 + a_4x^4 \neq 0\},
\]
and $|V( \varGamma(R))| = 2^4 - 1 = 15.$ For a nonzero zero-divisor $f$, define
\[
\operatorname{mindeg}(f)
= \min\{k\in\{1,2,3,4\} : \text{the coefficient of }x^k\text{ in }f\text{ is nonzero}\}.
\]
This partitions $V( \varGamma(R))$ into four subsets $V_1,V_2,V_3,V_4$, where $V_1$ is set of elements of minimal degree $1$ (coefficient of $x$ is nonzero). Here $a_1=1$, and $a_2,a_3,a_4\in\{0,1\}$ arbitrary, so $|V_1|=2^3=8$, and 
	\[
	\begin{aligned}
		V_1 = \{&
		x,
		x+x^2,
		x+x^3,
		x+x^4,
		x+x^2+x^3,
		x+x^2+x^4,\\
		&x+x^3+x^4,
		x+x^2+x^3+x^4
		\}.
	\end{aligned}
	\]
	Similarly, $V_2$ consists of elements with a minimal degree of $2$, as the coefficient of $x$ is $0$ and the coefficient of $x^2$ is non-zero. Let $a_1=0$, $a_2=1$, and $a_3,a_4\in\{0,1\}$. 
	Therefore, the size of \(V_2\) is \(2^2=4\), and \(V_2\) is defined as \(\{x^2, x^2+x^3, x^2+x^4, x^2+x^3+x^4\}\).
	Moreover, \(V_3\) includes all elements that have a minimum nonzero degree of \(3\).
	The coefficients of \(x\) and \(x^2\) must equal zero, but the coefficient of \(x^3\) must equal \(1\). The residual coefficient \(a_4\) can assume a value of either \(0\) or \(1\). Therefore, $V_3=\{x^3, x^3+x^4\}$. Moreover, \(V_4\) includes all elements with a minimum degree of \(4\). The coefficients of \(x\), \(x^2\), and \(x^3\) are all zero, while the coefficient of \(x^4\) is not zero.
	 Here $a_1=a_2=a_3=0$, $a_4=1$, so $|V_4|=1$, and 	$V_4 = \{x^4\}.$  Thus, $f,g$ are two nonzero zero-divisors with 
\[
\operatorname{mindeg}(f) = i,\quad \operatorname{mindeg}(g) = j,\quad i,j\in\{1,2,3,4\},
\]
then $fg=0  $ if $ i+j \ge 5.$
So vertices $f,g$ are adjacent in $ \varGamma(R)$ if and only if $i+j\ge 5$. From this, it follows that: $V_1$ is an independent set of order $8$, each vertex is  adjacent to $V_{4}=\{x^4\}$,  $V_2$ is an independent set of $4$ vertices, each joined to both vertices of $V_3$
	and to $x^4$, $V_3\cup V_4 = \{x^3,\ x^3+x^4,\ x^4\}$ induces a $K_3$. Also, there are no edges between $V_1$ and $V_2$, and no edges between $V_1$ and $V_3$.  The unique vertex $x^4\in V_4$ is adjacent to \emph{every} other vertex of $ \varGamma(R)$,
	so it is a universal vertex. The graph $ \varGamma(R)$ is shown in Figure \ref{fig 3}.
\begin{figure}[H]
	\centering
	\begin{tikzpicture}[
		scale=1.0,
		vertex/.style={circle, draw, inner sep=1.4pt},
		edge/.style={thin}
		]
		
		\node[vertex, fill=blue] (x1) at (-3, 3) {};
		\node[vertex, fill=blue] (x2) at (-2, 3) {};
		\node[vertex, fill=blue] (x3) at (-1, 3) {};
		\node[vertex, fill=blue] (x4) at ( 0, 3) {};
		\node[vertex, fill=blue] (x5) at ( 1, 3) {};
		\node[vertex, fill=blue] (x6) at ( 2, 3) {};
		\node[vertex, fill=blue] (x7) at ( 3, 3) {};
		\node[vertex, fill=blue] (x8) at ( 4, 3) {};
		
		\node[vertex, fill=green] (y1) at (-1, 1)   {};   
		\node[vertex, fill=red] (y2) at (-3, -1) {};   
		\node[vertex, fill=red] (y3) at ( 1, -.65) {};   
		
		\foreach \i/\j in {y1/y2,y1/y3,y2/y3}
		{
			\draw[edge] (\i) -- (\j);
		}
		
		\node[vertex, fill=pink] (z1) at (4,  2) {};
		\node[vertex, fill=pink] (z2) at (4,  1) {};
		\node[vertex, fill=pink] (z3) at (4, 0) {};
		\node[vertex, fill=pink] (z4) at (4, -1) {};
		
		\foreach \z in {z1,z2,z3,z4}
		{
			\foreach \y in {y1,y2,y3}
			{
				\draw[edge] (\z) -- (\y);
			}
		}
		
		\foreach \x in {x1,x2,x3,x4,x5,x6,x7,x8}
		{
			\draw[edge] (\x) -- (y1);
		}
		
	\end{tikzpicture}
	\caption{Zero--divisor graph $ \varGamma(R)$ for $R \cong \mathbb{Z}_2[x]/\langle x^5\rangle$.}
	\label{fig 3}
\end{figure}
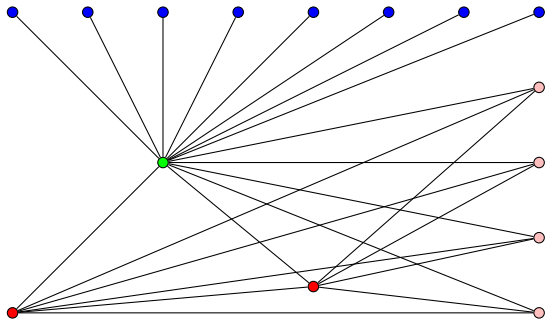

The following result gives the $A_\alpha$--spectrum of $ \varGamma\bigl(\mathbb{F}_p[x]/(x^{2b+1})\bigr)$. 
\begin{theorem}\label{spectra thm odd}
	Let $p$ be a prime and $b\ge 1$ a positive integer, and  let $ \varGamma(R)$ be the zero--divisor graph of $R = \mathbb{Z}_p[x]/\langle x^{2b+1}\rangle \cong \mathbb{F}_p[x]/(x^{2b+1}).$ Then the $A_\alpha$--spectrum of $ \varGamma(R)$ consist of the eigenvalue $ \alpha\bigl(p^{ i}-1\bigr),~ 1\le i\le b$, and the eigenvalue $ \alpha\bigl(p^{ i}-1\bigr) - 1,~b+1\le i\le 2b$ with multiplicity $n_i - 1 = (p-1)p^{ 2b-i} - 1.$ The remaining $2b$ eigenvalues are the eigenvalues of the following $2b\times 2b$ matrix $B=(B_{ij})_{1\le i,j\le 2b}$,
		where	$B = \alpha D^\ast + (1-\alpha)Q,$	with	$D^\ast = \operatorname{diag}(d_1,\dots,d_{2b})$
		and $Q=(Q_{ij})$ the quotient matrix of $A( \varGamma(R))$ with respect to the partition
		$V( \varGamma(R))=\bigcup_{i=1}^{2b}V_i$, given by
		\[
		Q_{ij} =
		\begin{cases}
			0, & \text{if } i=j,\ 1\le i\le b,\\[3pt]
			n_i - 1, & \text{if } i=j,\ b+1\le i\le 2b,\\[3pt]
			n_j, & \text{if } i\ne j,\ i+j\ge 2b+1,\\[3pt]
			0, & \text{if } i\ne j,\ i+j\le 2b.
		\end{cases}
		\]
		Equivalently,
		\[
		B_{ij} =
		\begin{cases}
			\alpha d_i, & \text{if } i=j,\ 1\le i\le b,\\[3pt]
			\alpha d_i + (1-\alpha)(n_i-1), & \text{if } i=j,\ b+1\le i\le 2b,\\[3pt]
			(1-\alpha)n_j, & \text{if } i\ne j,\ i+j\ge 2b+1,\\[3pt]
			0, & \text{if } i\ne j,\ i+j\le 2b,
		\end{cases}
		\]
		where	for 	$1\le i\le 2b$,
	\[
	d_i = \deg(v) =
	\begin{cases}
		p^{ i} - 1, & \text{if } 1\le i\le b,\\[3pt]
		p^{ i} - 2, & \text{if } b+1\le i\le 2b,
	\end{cases}
	\qquad v\in V_i.
	\]
\end{theorem}
\begin{proof}
	For  $a=2b+1$, and by Theorem \ref{odd case}, the vertex set of $ \varGamma(R)$ is partitioned into $V_i$ ($1\le i\le 2b$)  with order $n_i = |V_i| = (p-1)p^{ 2b-i}.$
	By the adjacency rule, for	$f\in V_i,\ g\in V_j, $ with $f\ne g$, a vertex  $f$ is adjacent to $g$  if $i+j\ge 2b+1.$  For $1\le i\le b$, there are no edges inside $V_i$, and a vertex in $V_i$ is adjacent to all vertices in $V_j$ with $j\ge 2b+1-i$. With $\sum_{j=2b+1-i}^{2b} |V_j| = p^{ i}-1,$	we get $\deg(v)=p^{ i}-1$ for $v\in V_i$. For $b+1\le i\le 2b$, $V_i$ induces a clique, and a vertex in $V_i$ is adjacent to all other vertices in $V_i$ and to all vertices in $V_j$ for $j\ge 2b+1-i$, $j\ne i$. The same sum gives $p^{ i}-1$ neighbors in $V_{2b+1-i}\cup\cdots\cup V_{2b}$, which now includes $V_i$, so we subtract $1$ for the vertex itself and obtain $\deg(v) = p^{ i}-2$ for $v\in V_i$. 	For each $i=1,\dots,2b$, define
	\[
	W_i = \Bigl\{x\in\mathbb{R}^{V( \varGamma(R))} :
	\mathrm{supp}(x)\subseteq V_i,\ \sum_{v\in V_i} x_v = 0\Bigr\},
	\]
	where  for $x = (x_v)_{v \in V( \varGamma(R))} \in \mathbb{R}^{V( \varGamma(R))},$ we have $\operatorname{supp}(x)= \{  v \in V( \varGamma(R)) : x_v \neq 0  \}.$ 	Thus, it is clear that $\dim W_i = |V_i|-1 = n_i -1$. We show that $W_i$ is invariant under $A_\alpha$ and compute the eigenvalue. If $x\in W_i$, then $x$ is supported on $V_i$ and $\sum_{v\in V_i}x_v=0$. Consider the action of the adjacency matrix $A$ on independent sets  $V_i$, for $1\leq i\leq b.$ 	For $v\in V_i$, there are no neighbors of $v$ in $V_i$, and $x$ vanishes on all other levels,
	so $(Ax)_v=0$.
	For $v\in V_k$ with $k\ne i$, if $i+k\le 2b$, there are no neighbors of $v$ in $V_i$. If $i+k\ge 2b+1$, then $v$ is adjacent to all vertices of $V_i$, and on defining condition of $W_i$, we obtain $(Ax)_v = \sum_{u\in V_i} x_u = 0.$ Thus, $Ax=0$. Similarly, $Dx=d_i x = (p^{ i}-1)x$ on $V_i$. Consequently, we have 
	$$ A_\alpha x = \alpha Dx + (1-\alpha) = \alpha(p^{ i}-1)x. $$
	Thus, every nonzero vector in $W_i$ is an eigenvector with corresponding eigenvalue $\lambda_i = \alpha(p^{ i}-1),$ for $1\le i\le b.$	Consider the action of the matrix $A$  on clique $V_i$, for $b+1\leq i\leq 2b. $ For $v\in V_i$, the neighbors of $v$ in the support of $x$ are the other vertices of $V_i$,	so we have 
	\[
	(Ax)_v = \sum_{\substack{u\in V_i\\ u\ne v}} x_u = \sum_{u\in V_i} x_u - x_v = -x_v.
	\]
	For $v\in V_k$ with $k\ne i$, either $i+k\le 2b$, in this case $v$ has no neighbors in $V_i$,
	or $i+k\ge 2b+1$, and in this case $v$ is adjacent to all vertices in $V_i$. So, it gives us $(Ax)_v = \sum_{u\in V_i} x_u = 0.$	Thus, $Ax=-x$ on its support, and $Dx=d_i x =(p^{ i}-2)x$ on $V_i$. So, we obtain
	\[
	A_\alpha x = \alpha Dx + (1-\alpha)Ax
	= \alpha(p^{ i}-2)x + (1-\alpha)(-x)
	= \bigl(\alpha(p^{ i}-1)-1\bigr)x.
	\]
	So, every nonzero vector in $W_i$ is an eigenvector with eigenvalue $\lambda_i = \alpha(p^{ i}-1)-1,$ for $b+1\le i\le 2b.$
	In either case $W_i$ is $A_\alpha$--invariant and contributes an eigenvalue $\lambda_i$ with
	multiplicity $\dim W_i = n_i-1$.  Summing over $i=1,\dots,2b$, the total dimension of $W= \bigoplus_{i=1}^{2b} W_i$	is
	\[
	\sum_{i=1}^{2b}(n_i-1)
	= \left(\sum_{i=1}^{2b}n_i\right) - 2b
	= (p^{ 2b}-1) - 2b
	= p^{ 2b} - (2b+1).
	\]
	Thus, $2b$ $A_{\alpha}$ eigenvalues are yet to be determined. 	Now, consider vectors that are constant on each subset $V_i$.
	Consider $U_0 = \{x\in\mathbb{R}^{V( \varGamma(R))} : x \text{ is constant on each }V_i\}, $ so that  $\dim U_0 = 2b$, and $\mathbb{R}^{V( \varGamma(R))} = W \oplus U_0,$	since
	\[
	\dim W + \dim U_0 = (p^{ 2b}-(2b+1)) + 2b = p^{ 2b}-1 = |V( \varGamma(R))|.
	\]
	As, every vertex in $V_i$ has the same number of neighbors in each $V_j$, so the partition $\bigcup_{i=1}^{2b}V_{i}$ is equitable, it implies that $U_0$ is invariant under $A$, and hence under $A_\alpha$. 	Let $x\in U_0$ and write $x_i$ for the common value of $x$ on $V_i$. Then $x$ is encoded by the column vector $\mathbf{x} = (x_1,\dots,x_{2b})^T \in \mathbb{R}^{2b}.$ For $v\in V_i$, we have $(Ax)_v = \sum_{j=1}^{2b} Q_{ij} x_j,$ where $Q_{ij}$ is the number of neighbors in $V_j$ of a given vertex in $V_i$. From the adjacency rule $i+j\ge 2b+1$ and the structure for $a=2b+1$, we have the following facts: If $1\le i\le b$, $V_i$ is independent, so $Q_{ii}=0$, and $v\in V_i$ is adjacent to every vertex in $V_j$ for $j\ge 2b+1-i$, hence $Q_{ii} = 0,$ and 
	$$Q_{ij} =
		\begin{cases}
			n_j, & j\ne i,\ j\ge 2b+1-i,\\
			0, & j \le 2b-i.
		\end{cases}$$
		For $b+1\le i\le 2b$, $V_i$ constitutes a clique, implying that each vertex $v\in V_i$ is adjacent to the remaining $n_i-1$ vertices within $V_i$ and to all vertices in $V_j$ for $j\ge 2b+1-i$, where $j\ne i$. So, $Q_{ii} = n_i-1,$ and 
		$$Q_{ij} =
		\begin{cases}
			n_j, & j\ne i,\ j\ge 2b+1-i,\\
			0, & j \le 2b-i.
		\end{cases}$$	
		Thus, on $U_0$, $Ax $ gives rise to $Q\mathbf{x}$, and $Dx$ is equivalent to $D^\ast \mathbf{x},$
	where $D^\ast = \operatorname{diag}(d_1,\dots,d_{2b})$. 
	By applying the definition $A_\alpha = \alpha D + (1-\alpha) A$, it follows that the operation on the subspace can be simplified to $\bigl(\alpha D^\ast + (1-\alpha)Q\bigr)\mathbf{x} = B\mathbf{x}$. Consequently, the matrix $B$ serves as the representation for the restriction of $A_\alpha$ to the subspace $U_0$ when expressed in the basis of indicator vectors $\{\mathbf{1}_{V_1}, \dots, \mathbf{1}_{V_{2b}}\}$. In this context, $\mathbf{1}_{V_i}$ denotes the characteristic vector for each respective set $V_i$. Given that $\dim U_0 = 2b$, the eigenvalues of this restricted operator correspond precisely to the simple eigenvalues $\{\theta_1, \dots, \theta_{2b}\}$ of the $2b \times 2b$ matrix $B$.
	 The matrices $D^{\ast},Q$ and $B$ are as in the statement. In matrix form, we write  $B=(B_{ij})_{1\le i,j\le 2b}$ in $2\times 2$ block form as $B=	\begin{pmatrix}
		B^{(11)} & B^{(12)}\\[4pt]
		B^{(21)} & B^{(22)}
	\end{pmatrix},$	where each block is of size $b\times b$. More precisely,  $B^{(11)}=
	\begin{pmatrix}
		\alpha d_1 & 0           & \cdots & 0\\
		0          & \alpha d_2  & \cdots & 0\\
		\vdots     & \vdots      & \ddots & \vdots\\
		0          & 0           & \cdots & \alpha d_b
	\end{pmatrix}$, and the top–right block has the form
	\[
	B^{(12)}=
	\begin{pmatrix}
		0              & 0              & \cdots & 0                  & (1-\alpha)n_{2b}\\[2pt]
		0              & 0              & \cdots & (1-\alpha)n_{2b-1} & (1-\alpha)n_{2b}\\[2pt]
		\vdots         & \vdots         & \ddots & \vdots             & \vdots\\[2pt]
		0              & (1-\alpha)n_{b+2} & \cdots & (1-\alpha)n_{2b-1} & (1-\alpha)n_{2b}\\[2pt]
		(1-\alpha)n_{b+1} & (1-\alpha)n_{b+2} & \cdots & (1-\alpha)n_{2b-1} & (1-\alpha)n_{2b}
	\end{pmatrix}.
	\]
	The bottom–left block is
	\[
	B^{(21)}=
	\begin{pmatrix}
		0              & 0              & \cdots & 0               & (1-\alpha)n_{b}\\[2pt]
		0              & 0              & \cdots & (1-\alpha)n_{b-1} & (1-\alpha)n_{b}\\[2pt]
		\vdots         & \vdots         & \ddots & \vdots          & \vdots\\[2pt]
		0              & (1-\alpha)n_{2} & \cdots & (1-\alpha)n_{b-1} & (1-\alpha)n_{b}\\[2pt]
		(1-\alpha)n_{1} & (1-\alpha)n_{2} & \cdots & (1-\alpha)n_{b-1} & (1-\alpha)n_{b}
	\end{pmatrix}.
	\]
	Finally, the bottom–right block is
	\begin{footnotesize}
		\[
	B^{(22)}=
	\begin{pmatrix}
		\alpha d_{b+1} + (1-\alpha)(n_{b+1}-1) & (1-\alpha)n_{b+2} & \cdots & (1-\alpha)n_{2b}\\[2pt]
		(1-\alpha)n_{b+1} & \alpha d_{b+2} + (1-\alpha)(n_{b+2}-1) & \cdots & (1-\alpha)n_{2b}\\[2pt]
		\vdots & \vdots & \ddots & \vdots\\[2pt]
		(1-\alpha)n_{b+1} & (1-\alpha)n_{b+2} & \cdots &
		\alpha d_{2b} + (1-\alpha)(n_{2b}-1)
	\end{pmatrix},
	\]
	\end{footnotesize}
	where for $1\le i\le 2b$, we have
	\[
	d_i =
	\begin{cases}
		p^{ i} - 1, & 1\le i\le b,\\[3pt]
		p^{ i} - 2, & b+1\le i\le 2b,
	\end{cases}
	\qquad
	n_i = (p-1)p^{ 2b-i}.
	\]
\end{proof}

We will illustrate Theorem \ref{spectra thm odd} by the following example, with $b=2$ and $p=2$.\\
\begin{example}\end{example}
	Let $ \varGamma(R)$ the zero--divisor graph of $R=\mathbb{Z}_2[x]/\langle x^5\rangle\cong \mathbb{F}_2[x]/(x^5)$. 
	 Then $V( \varGamma(R))=V_1 \cup V_2 \cup V_3 \cup V_4,$	where	$|V_1|= 2^3 = 8, |V_2|= 4,|V_3| = 2, |V_4| = 1,$ and $|V( \varGamma(R))|=2^4 -1=15$. By Theorem~\ref{odd case} and Theorem~\ref{spectra thm odd} (with $p=2$, $b=2$),
	the degrees of vertices in each $V_i$ are:	$d_1 = 1, d_2 = 3, d_3 = 6, $ and $ d_4 = 14.$ 	Thus, for $p=2$, $b=2$, we have the following eigenvalues
	\[
	\begin{aligned}
		\lambda_1 &= \alpha(2^1-1) = \alpha, 
		&\quad \operatorname{mult}(\lambda_1) &= |V_1|-1 = 8-1 = 7,\\
		\lambda_2 &= \alpha(2^2-1) = 3\alpha, 
		&\quad \operatorname{mult}(\lambda_2) &= |V_2|-1 = 4-1 = 3,\\
		\lambda_3 &= \alpha(2^3-1)-1 = 7\alpha - 1,
		&\quad \operatorname{mult}(\lambda_3) &= |V_3|-1 = 2-1 = 1,\\
		\lambda_4 &= \alpha(2^4-1)-1 = 15\alpha - 1,
		&\quad \operatorname{mult}(\lambda_4) &= |V_4|-1 = 1-1 = 0,
	\end{aligned}
	\]
	where $\operatorname{mult}(\lambda_i)$ is the multiplicity of the eigenvalue $ \lambda_{i}.$ 	Hence $\lambda_4$ does not actually appear in the spectrum as  its multiplicity is zero. The remaining $2b=4$ eigenvalues arise from the restriction of $A_\alpha$ to the
	subspace $U_0$, which has dimension $\dim U_0 = 4$.  As in Theorem~\ref{spectra thm odd}, this
	restriction is represented (in the basis $\{\mathbf{1}_{V_1},\dots,\mathbf{1}_{V_4}\}$)
	by the $4\times 4$ matrix
	\[
	B = \alpha D^\ast + (1-\alpha)Q,
	\]
	where $D^\ast = \operatorname{diag}(d_1,d_2,d_3,d_4)$ and $Q=(Q_{ij})$ is the
	quotient matrix of the adjacency matrix with respect to the partition
	$V_1,\dots,V_4$. Also, $D^\ast = \operatorname{diag}(1,3,6,14)$, the matrices  $Q$ and $B$ are given below:
	\[
	Q =
	\begin{pmatrix}
		0 & 0 & 0 & 1\\
		0 & 0 & 2 & 1\\
		0 & 4 & 1 & 1\\
		8 & 4 & 2 & 0
	\end{pmatrix},
	\quad B =
	\begin{pmatrix}
		\alpha & 0 & 0 & 1-\alpha\\[3pt]
		0 & 3\alpha & 2(1-\alpha) & 1-\alpha\\[3pt]
		0 & 4(1-\alpha) & 5\alpha+1 & 1-\alpha\\[3pt]
		8(1-\alpha) & 4(1-\alpha) & 2(1-\alpha) & 14\alpha
	\end{pmatrix}.
	\]
	Let $\theta_1,\theta_2,\theta_3,\theta_4$ be the eigenvalues of $B$. Then the $A_\alpha$--spectrum of
	$ \varGamma\bigl(\mathbb{F}_2[x]/(x^5)\bigr)$ is
	\[
	\bigl\{
	\underbrace{\alpha,\dots,\alpha}_{7},
	\underbrace{3\alpha,\dots,3\alpha}_{3},
	7\alpha-1,
	\theta_1,\theta_2,\theta_3,\theta_4
	\bigr\}.
	\]

From Theorem \ref{spectra thm odd}, we have the following consequence.
	\begin{corollary}\label{cor:adj-L-Q-spectra}
		Let $p$ be a prime and $b\ge 1$ a positive integer, and the notation of
		Theorem~\ref{spectra thm odd},	$n_i=(p-1)p^{ 2b-i}$ and $D^\ast=\operatorname{diag}(d_1,\dots,d_{2b})$,
		and let $Q$ be the quotient matrix of $A( \varGamma(R))$ with respect to
		the partition $\{V_1,\dots,V_{2b}\}$.
		Then the following hold.
		\begin{enumerate}
			\item The adjacency spectrum of $ \varGamma(R)$ is given by
			\[
		\bigcup_{i=1}^{b}\{0^{[n_i-1]}\}
			\cup
			\bigcup_{i=b+1}^{2b}\{(-1)^{[n_i-1]}\}
			\cup\sigma(Q),
			\]
			where $\sigma(Q)$ is the spectrum of $Q.$
			
			\item The signless Laplacian spectrum of $ \varGamma(R)$ if given by
			\[
			\bigcup_{i=1}^{b}\{(p^{ i}-1)^{[n_i-1]}\}
			\cup
			\bigcup_{i=b+1}^{2b}\{(p^{ i}-3)^{[n_i-1]}\}
			\cup\sigma(D^\ast+Q),
			\]
			where $\sigma(D^\ast+Q)$ is the spectrum of  $D^\ast+Q=2B\!\left(\tfrac12\right)$.
		\end{enumerate}
	\end{corollary}
	\begin{proof} Let $\lambda_{i}(\alpha)$ be the $A_{\alpha}$ eigenvalue of $ \varGamma(R)$. Then by Theorem \ref{spectra thm odd}, we have the following cases.\\ 
		(1) The adjacency matrix of $G$ is $A(G)=A_0(G)$. So, with $\alpha=0$ in Theorem~\ref{spectra thm odd}, we obtain
		\[
		\lambda_i(0)=0,\quad 1\le i\le b,\qquad\text{and}\qquad
		\lambda_i(0)=-1,\quad b+1\le i\le 2b,
		\]
		each with multiplicity $n_i-1$. Theorem~\ref{spectra thm odd} further shows that the other \(2b\) eigenvalues of \(A_0(G)=A(G)\) coincide with the eigenvalues of  
		\[
		B(0)=0\cdot D^{\ast}+(1-0)Q^{\ast}=Q^{\ast},
		\]
		that is, the quotient matrix of \(A(G)\) corresponding to the partition \(\{V_i\}\).
		
		\medskip
		
		\noindent
		(2) The signless Laplacian of \(G\) is defined by
		\[
		Q(G)=D(G)+A(G)=2A_{1/2}(G).
		\]
		Hence, if \(\mu\) is an eigenvalue of \(A_{1/2}(G)\), then \(2\mu\) is an eigenvalue of \(Q(G)\) with equal multiplicities. 
		For for $\alpha=\tfrac12$. Theorem~\ref{spectra thm odd} implies that
		 $\lambda_i\Bigl(\frac12\Bigr) = \tfrac12\bigl(p^{ i}-1\bigr),$ for $1\le i\le b,$ and 
		$\lambda_i\Bigl(\frac12\Bigr) = \tfrac12\bigl(p^{ i}-1\bigr)-1 = \tfrac{p^{ i}-3}{2},$ for $b+1\le i\le 2b,$ 	each with multiplicity $n_i-1$. 
		Hence, the corresponding eigenvalues of $Q(G)=2A_{1/2}(G)$ are
		\[
		p^{ i}-1,\quad 1\le i\le b,\quad\text{and}\quad
		p^{ i}-3,\quad b+1\le i\le 2b,
		\]
		each with multiplicity $n_i-1$.	The remaining $2b$ eigenvalues of $A_{1/2}(G)$ are the eigenvalues of the following matrix
		\[
		B\Bigl(\frac12\Bigr) = \frac12 D^\ast + \frac12 Q^\ast
		= \frac12\bigl(D^\ast+Q^\ast\bigr).
		\]
		Multiplying by $2$, the remaining $2b$ eigenvalues of $Q( \varGamma(R))$ are exactly the eigenvalues of $D^\ast+Q^\ast$. This proves part (2).
	\end{proof}
	
	The following consequence of Theorem \ref{spectra thm odd} shows that the Laplacian eigenvalues of $ \varGamma(R)$ are integers.
	\begin{corollary}\label{L integral odd}
		The Laplacian matrix $L( \varGamma(R))$ is integral, and its spectrum is
		\[
		\{0\}\cup\bigcup_{\substack{1\le i\le 2b\\ i\neq b}}\{(p^{ i}-1)^{[n_i]}\}\cup\{(p^{ b}-1)^{[n_b-1]}\},
		\]
		where $n_i = |V_i| = (p-1)p^{2b-i}$.
	\end{corollary}
	\begin{proof}
		We continue with the notation of Theorem~\ref{spectra thm odd}. Recall that $V(G)=\bigcup_{i=1}^{2b}V_i, n_i=|V_i|=(p-1)p^{2b-i},$ and for $v\in V_i$, we have
		\[
		d_i=\deg(v) =
		\begin{cases}
			p^{ i}-1, & 1\le i\le b,\\
			p^{ i}-2, & b+1\le i\le 2b.
		\end{cases}
		\]
		It is known that
		\[
		\frac{1}{\alpha-\beta}\bigl(A_\alpha(G)-A_\beta(G)\bigr)
		= D(G) - A(G) = L(G).
		\]
		For $1\leq i\leq b$, we get the Laplacian eigenvalues 
		$$\frac{1}{\alpha-\beta}\Big( \alpha(p^{ i}-1)-( \beta(p^{ i}-1))\Big)=p^{i}-1,$$
		 with multiplicity $n_{i}-1.$ Similarly, for $b+1\leq i\leq 2b$, we have the Laplacian eigenvalues
		 $$\frac{1}{\alpha-\beta}\Big( \alpha(p^{ i}-1)-1-( \beta(p^{ i}-1))-1\Big)=p^{i}-1,$$ each with multiplicity $n_{i}-1.$ Let $D^\ast=\operatorname{diag}(d_1,\dots,d_{2b})$ be the diagonal matrix, and let $Q^\ast=(Q_{ij})$ be the quotient matrix of $A(G)$ with respect to $\{V_i\}$ as in Theorem~\ref{spectra thm odd}. Then the quotient matrix of $A_\alpha(G)$ is
		\[
		B(\alpha) = \alpha D^\ast + (1-\alpha)Q^\ast.
		\]
		For $\alpha=1$ and $\alpha=0,$ we have $B(1)=D^\ast, $ and $ B(0)=Q^\ast,$ so, for any $\alpha\neq\beta$, where
		\[
		\frac{1}{\alpha-\beta}\bigl(B(\alpha)-B(\beta)\bigr)
		= D^\ast - Q^\ast,
		\]
		where  $D^\ast-Q^\ast$ is the quotient matrix of $L(G)$ with respect to the partition $\{V_i\}$. Denote this $2b\times 2b$ matrix by $	\overline{L}= D^\ast - Q^\ast,$ with
		\[
		\overline{L}=
		\begin{pmatrix}
			d_1 & 0   & 0   & \cdots & 0         & 0         & -n_{2b} \\[3pt]
			0   & d_2 & 0   & \cdots & 0         & -n_{2b-1} & -n_{2b} \\[3pt]
			0   & 0   & d_3 & \cdots & -n_{2b-2} & -n_{2b-1} & -n_{2b} \\[3pt]
			\vdots & \vdots & \vdots & \ddots & \vdots    & \vdots    & \vdots \\[3pt]
			0   & 0   & -n_3 & \cdots &
			p^{ 2b-2}-1-n_{2b-2} & -n_{2b-1} & -n_{2b} \\[3pt]
			0   & -n_2 & -n_3 & \cdots & -n_{2b-2} &
			p^{ 2b-1}-1-n_{2b-1} & -n_{2b} \\[3pt]
			-n_1 & -n_2 & -n_3 & \cdots & -n_{2b-2} & -n_{2b-1} &
			p^{ 2b}-1-n_{2b}
		\end{pmatrix},
		\]
		where, for \(1\le i\le 2b\), we have
		\[
		d_i =
		\begin{cases}
			p^{i}-1, & 1\le i\le b,\\[3pt]
			p^{i}-2, & b+1\le i\le 2b,
		\end{cases}
		\quad\text{and}\quad
		n_i = (p-1)p^{2b-i}.
		\]
		Next, we calculate the  eigenvalues of $\overline{L}$.	 It is known that $0$ is the eigenvalue of $\overline{L}, $ due to positive semidefinite property of the Laplacian matrix of a  connected $ \varGamma(R).$
	For $r\ge1$, we recall the tail identity,
	\begin{align}
		\sum_{j=2b-r+1}^{2b} n_j
		&= \sum_{j=2b-r+1}^{2b} (p-1)p^{2b-j}
		= (p-1)\sum_{t=0}^{r-1} p^{t}
		= p^{r}-1. \label{eq:tail-sum}
	\end{align}
	For the eigenvalues $p^{k}-1$ of $\overline{L}$, for $1\le k\le b-1$, we fix $k$ with $1\le k\le b-1$, and set
	$	\lambda_k = p^{k}-1.$ Now, consider a  vector $v^{(k)}\in\mathbb{R}^{2b}$ defined as
	\begin{equation}\label{eq:v-k-def}
		v^{(k)}_j =
		\begin{cases}
			0, & 1\le j<k,\\
			- \varGamma_k, & j=k,\\
			1, & k<j\le 2b-k,\\
			0, & 2b-k<j\le 2b,
		\end{cases}
	\end{equation}
	for some scalar $ \varGamma_k$ to be chosen.	We claim that we can choose $ \varGamma_k$ so that $\overline{L} v^{(k)} = \lambda_k v^{(k)}.$

	For rows of $\overline{L}$, if $i<k$, then $v^{(k)}_i=0$, and all nonzero off–diagonal entries in row $i$ are in columns $	j\in\{2b-i+1,\dots,2b\},$	which are strictly greater than $2b-k$ because $i<k$. Thus $v^{(k)}_j=0$ for all such columns, and hence
	\[
	(\overline{L}v^{(k)})_i
	= d_i v^{(k)}_i - \sum_j n_j v^{(k)}_j = 0
	= \lambda_k v^{(k)}_i.
	\]
	For the $k$-th row of $\overline{L}$, the off–diagonal $-n_j$'s again exists only in columns $2b-k+1,\dots,2b$, where $v^{(k)}_j=0$, so we obtain
	\[
	(\overline{L}v^{(k)})_k = d_k v^{(k)}_k = (p^k-1) v^{(k)}_k
	= \lambda_k v^{(k)}_k.
	\]
	Thus the eigenvalue equation holds for $i=k$ independently of $ \varGamma_k$.
	For the middle rows $k<i\le 2b-k$, the only nonzero entries of $v^{(k)}$ visible to row $i$ are some of the $1$'s (from the block $k<j\le 2b-k$) and possibly the entry $- \varGamma_k$ at position $j=k$.	Using again the geometric-series identity \eqref{eq:tail-sum} to evaluate the sums of $n_j$ against the block of $1$'s, we find  that the contribution of the $1$-block always sums to exactly $p^{i}-1$, matching the diagonal part. The terms involving $- \varGamma_k$ cancel between the diagonal and off–diagonal contributions. Thus, we have
	\[
	(\overline{L}v^{(k)})_i = (p^{i}-1)v^{(k)}_i = \lambda_k v^{(k)}_i,
	\]
	for all $k<i\le 2b-k$, again without imposing any restriction on $ \varGamma_k$.
	
	For bottom rows $i>2b-k$, the only rows that constrain $ \varGamma_k$ are the bottom ones with $i>2b-k$. 
	A direct calculation, utilizing the explicit representation of these rows and the pattern \eqref{eq:v-k-def}, demonstrates that all such rows provide the same scalar condition.
	\[
	 \varGamma_k p^{ 2b+1-k} -  \varGamma_k p^{ 2b-k} - p^{ 2b-k} + p^{ k} = 0.
	\]
	Upon analyzing the earlier equation for $ \varGamma_k$, we derive
	\begin{equation}\label{eq:gamma-k-small}
		 \varGamma_k = \frac{p^{2b-k}-p^k}{p^{2b-k}(p-1)} = \frac{1-p^{2k-2b}}{p-1}.
	\end{equation}
	With this choice, every row of $\overline{L}v^{(k)}=\lambda_k v^{(k)}$ holds, so $p^k-1$ is an eigenvalue for $1\le k\le b-1$.
	
	\medskip
	
	Now, for the eigenvalues $p^{k}-1$ of $\overline{L}$, with $b+1\le k\le 2b$, we fix $k$ with $b+1\le k\le 2b$, and set $\lambda_k=p^k-1$. With $m=2b+1-k$, consider
	\begin{equation}\label{eq:w-k-def}
		w^{(k)}_j =
		\begin{cases}
			0,          & 1\le j < m,\\
			- \varGamma_k,  & m\le j \le k-1,\\
			1,          & j=k,\\
			0,          & k<j\le 2b,
		\end{cases}
	\end{equation}
	for another scalar $ \varGamma_k$ to be chosen. We are required to show that $\overline{L} w^{(k)} = \lambda_k w^{(k)}.$ The rows of $\overline{L}$, whose nonzero entries lie entirely in the zero part of $w^{(k)}$ trivially satisfy $0=0$. For rows that see the constant block of $- \varGamma_k$ and the single $1$ at position $k$, a direct computation using geometric–series sums \eqref{eq:tail-sum} of the $n_j$ gives that the eigenvalue equation reduces to the scalar condition $ \varGamma_k p^{ 2(k-b)} -  \varGamma_k p - p + 1 = 0.$ Its evaluation gives
	\begin{equation}\label{eq:gamma-k-large}
		 \varGamma_k = \frac{p-1}{p^{2(k-b)}-p}
		= \frac1p\cdot\frac{1}{1+p+\cdots+p^{2(k-b)-2}}.
	\end{equation}
	With the above choice of $ \varGamma_k$, each row satisfies $(\overline{L}w^{(k)})_i = \lambda_k w^{(k)}_i,$ and 
	hence if follows that
	\[
	\overline{L}w^{(k)} = (p^k-1)w^{(k)},\qquad b+1\le k\le 2b.
	\]
	The numbers $p^k-1$ are strictly increasing in $k$, so all these eigenvalues are distinct, and together with $0$ they give $2b$ eigenvalues for the $2b\times 2b$ matrix $\overline{L}$.
\end{proof}

\section{Comments}
The following result due to Aouchiche and Hensen related the Laplacian eigenvalues  with the distance Laplacian eigenvalues of a connected graph with diameter $2.$
\begin{lemma}[\cite{Ah2013}]\label{ah lemma}
	Let $G$ be a connected graph on $n$ vertices with diameter $D \leq 2$. Let $\lambda^L_1 \geq \lambda^L_2 \geq \cdots \geq \lambda^L_{n-1} > \lambda^L_n = 0$ be the Laplacian spectrum of $G$. Then the distance Laplacian spectrum of $G$ is $2n - \lambda^L_{n-1} \geq 2n - \lambda^L_{n-2} \geq \cdots \geq 2n - \lambda^L_1 > \partial^L_n = 0$. Moreover, for every $i \in \{1, 2, \ldots, n-1\}$ the eigenspaces corresponding to $\lambda^L_i$ and to $2n - \lambda^L_{n-i}$ are the same.
\end{lemma}

From Corollary \ref{L integral even}, the Laplacian spectrum of $ \varGamma(\mathbb{Z}_p[x]/\langle x^{2b}\rangle)$ is
	\[\{0\}
	\cup
	\bigcup_{\substack{1\le i\le 2b-1\\ i\neq b}}
	\bigl\{(p^{ i}-1)^{[n_i]}\bigr\}
	\cup
	\bigl\{(p^{ b}-1)^{[n_b-1]}\bigr\},
	\]
	where $n_i = |V_i| = (p-1)p^{ 2b-1-i}, 1\le i\le 2b-1,$ and the total number of vertices is $n = |V( \varGamma(R))|= p^{2b-1}-1.$ By Lemma~\ref{ah lemma}, with Laplacian spectrum $\sigma(L( \varGamma(R))),$ the distance Laplacian spectrum is
	\[
	\{ 2n-\lambda : \lambda\in\sigma(L( \varGamma(R)))\setminus\{0\} \}
	\cup\{0\},
	\]
	with the same multiplicities as the corresponding Laplacian eigenvalues (and one zero eigenvalue). Since $2n = 2(p^{2b-1}-1)=2p^{2b-1}-2$, we obtain the distance Laplacian spectrum $ \varGamma(R)$ as	
	\[
	\{0\}
	\cup
	\{ 2n-0 \}
	\cup
	\bigcup_{\substack{1\le i\le 2b-1\\ i\neq b}}
	\bigl\{ (2n-(p^{ i}-1))^{[n_i]}\bigr\}
	\cup
	\bigl\{ (2n-(p^{ b}-1))^{[n_b-1]}\bigr\}.
	\]
	With $2n=2p^{2b-1}-2$,  the explicit form is
	\[
	\{0\}
	\cup
	\{ 2p^{2b-1}-2 \}
	\cup
	\bigcup_{\substack{1\le i\le 2b-1\\ i\neq b}}
	\bigl\{ (2p^{2b-1}-p^{ i}-1)^{[n_i]}\bigr\}
	\cup
	\bigl\{ (2p^{2b-1}-p^{ b}-1)^{[n_b-1]}\bigr\}.
	\]

	Now, for $ \varGamma(\mathbb{Z}_p[x]/\langle x^{2b+1}\rangle)$, order is $n=p^{2b}-1,$ and $2n = 2(p^{2b}-1) = 2p^{2b}-2.$ By the Corollary \ref{L integral odd}, the Laplacian spectrum  of $ \varGamma(R)$ is
	\[\{0\}\cup\bigcup_{\substack{1\le i\le 2b\\ i\neq b}}\{(p^{ i}-1)^{[n_i]}\}
	\cup\{(p^{ b}-1)^{[n_b-1]}\}.
	\]
	With Lemma \ref{ah lemma}, $	\mu_i = 2n - \lambda_i,$  the distance Laplacian spectrum of  $ \varGamma(R))$ is
	\[
	\{ 2n-0 \}
	\cup
	\bigcup_{\substack{1\le i\le 2b\\ i\neq b}}
	\bigl\{ (2n - (p^{ i}-1))^{[n_i]}\bigr\}
	\cup
	\bigl\{ (2n - (p^{ b}-1))^{[n_b-1]}\bigr\}.
	\]
	By substituting $2n = 2p^{2b}-2$, we derive the exact distance Laplacian spectrum of $ \varGamma(R)$ as
	\[ 	\{ 2p^{2b}-2 \} \cup \bigcup_{\substack{1\le i\le 2b\\ i\neq b}} 	\bigl\{ (2p^{2b}-p^{ i}-1)^{[n_i]}\bigr\} \cup 	\bigl\{ (2p^{2b}-p^{ b}-1)^{[n_b-1]}\bigr\}.
	\]
	
	Thus, we have the following result related to the distance Laplacian integral eigenvalues of $ \varGamma(\mathbb{Z}_p[x]/\langle x^{c}\rangle)$.
	\begin{corollary}\label{cor distance}
		The zero-divisor graph $ \varGamma(\mathbb{Z}_p[x]/\langle x^{b}\rangle)$ is distance Laplacian integral.
	\end{corollary}

\section*{Conclusion}
In this paper, we investigated the zero-divisor graph $ \varGamma\!\left(\mathbb{Z}_{p}[x]/\langle x^{c} \rangle\right)$ associated with the commutative ring $\mathbb{Z}_{p}[x]/\langle x^{c} \rangle$, where $p$ is a prime and $c\in\mathbb{N}$. We explicitly determined the spectrum of the $A_{\alpha}$-matrix of these graphs in Theorems \ref{spectra thm even} and \ref{spectra thm odd}, and as important special cases obtained the adjacency and signless Laplacian spectra. Furthermore, we proved that the Laplacian matriix of $ \varGamma\!\left(\mathbb{Z}_{p}[x]/\langle x^{c} \rangle\right)$ have integral eigenvalues (Corollaries \ref{L integral odd} and \ref{L integral odd}), thereby showing that these zero-divisor graphs are Laplacian integral. As a consequence for graph of diameter $2$, Corollary \ref{cor distance} shows that $ \varGamma\!\left(\mathbb{Z}_{p}[x]/\langle x^{c} \rangle\right)$ is  distance Laplacian integral. These results highlight a strong interplay between the algebraic structure of the underlying ring and the spectral properties of its zero-divisor graph, and provide a foundation for further spectral analysis of zero-divisor graphs arising from other families of commutative rings.

	\section*{Declarations}
	\noindent \textbf{Data Availability:}	There is no data associated with this article as data sharing is not applicable since no data sets were generated or analyzed during the current study.
	
	\noindent \textbf{Funding:} The authors declare that no funds, grants, or other support were received during the preparation of this manuscript.
	
	\noindent \textbf{Conflict of interest:} The authors have no competing interests to declare that are relevant to the content of this article.
	
	\noindent\textbf{Note:} I welcome any comments and suggestions regarding this article; please feel free to contact
	me at \href{bilalahmadrr@gmail.com}{bilalahmadrr@gmail.com}
	
	

\begin{thebibliography}{0}
		\bibitem{AAlphaBounds}
		A.~Alhevaz, M.~Baghipur and H.~A.~Ganie,
		Bounds for the spectral radius of the $A_\alpha$-matrix of graphs,
		\emph{Indian J.\ Pure Appl.\ Math.} \textbf{55} (2024), 298--309.
		\bibitem{al} { D.F. Anderson and P.S. Livingston, The zero-divisor graph of a commutative ring, \em J. Algebra} {\bf 217} (1999) 434--447.
		\bibitem{Ah2013} M. Aouchiche, and  P. Hansen, Two Laplacians for the distance matrix of a graph, \emph{Linear Algebra  Appl.} \textbf{439}(1) (2013)  21--33.
		
		\bibitem{aouchiche-hansen-distance} M. Aouchiche and  P. Hansen,  Distance spectra of graphs: a survey,  \textit{Linear Algebra Appl.} {\bf 458} (2014) 301--386.
		\bibitem{AH3} M. Aouchiche and  P. Hansen,  Some properties of the distance Laplacian eigenvalues a graph,  \textit{Czech. Math. J.} {\bf 64,139} (2014) 751--761.
		
		\bibitem{ib}  {I. Beck, Coloring of a commutative rings, \em J. Algebra} {\bf 116} (1988) 208--226.
		
		
		
		
		\bibitem{brouwer-haemers}
		A.~E.~Brouwer and W.~H.~Haemers,
		\newblock {\em Spectra of Graphs},
		\newblock Springer, New York, 2012.
		
		\bibitem{cds} D. M. Cvetkovi\'{c}, P. Rowlison and S. Simi\'c, \textit{An Introduction to Theory of Graph spectra, Spectra of graphs, Theory and application}, London Math. S. Student Text, 75, Cambridge University Press, Inc. UK, 2010.
		\bibitem{filipovski} S. Filipovski and R. Jajcay, Bounds for the energy of graphs, \textit{Math.} \textbf{9}(14) (2021) 1687.
		
		\bibitem{gutman2010} I. Gutman, Hyperenergetic and Hypoenergetic Graphs, Zbornik Radova. In Selected Topics on Applications of Graph Spectra; Matematicki Institut SANU: Beograd, Serbia, 2011, pp. 113--135.
		\bibitem{harary-schwenk}
		F.~Harary and A.~J.~Schwenk,
		\newblock Which graphs have integral spectra?
		\newblock in: {\em Graphs and Combinatorics} (R.~A.~Bari and F.~Harary, eds.),
		\newblock Springer, Berlin, 1974, pp.~45--51.
		\bibitem{fareeha} F. Jamal and M. Imran, Distance spectrum of some zero-divisor graphs, \emph{Aims Math.} \textbf{9}(9) (2024) 23979--23996.
		\bibitem{johnson} C. Johnson and  R. Sankar, Graph energy and topological descriptors of zero-divisor graph associated with the commutative ring, \emph{J. Appl. Math. Comp.} \textbf{69} (2023) 2641--2656.
		\bibitem{klotz-sander}
		W.~Klotz and T.~Sander,
		\newblock Some properties of graphs with integral spectrum,
		\newblock {\em Electron. J. Combin.} {\bf 19} (2012), \#P20.
		\bibitem{shi} X. Li, Y. Shi and I. Gutman, \textit{Graph Energy}, Springer, New York (2012).
		\bibitem{LinXueShu2018}
		H.~Lin, J.~Xue and J.~Shu,
		\newblock On the $A_{\alpha}$-spectra of graphs,
		\newblock {\em Linear Algebra and its Applications} {\bf 550} (2018), 105--120.
		
		\bibitem{monius} K. Mönius, Eigenvalues of zero-divisor graphs of finite commutative rings, \textit{J. Algebra Comb.} \textbf{54} (2021) 787--802.
		
		\bibitem{nv} V. Nikiforov, Beyond graph energy: Norms of graphs and matrices, \emph{Linear Algebra Appl.} \textbf{506} (2016) 82--138.
		\bibitem{nikiforov-aalpha}
		V.~Nikiforov,
		\newblock Merging the $A$– and $Q$–spectral theories,
		\newblock {\em Appl. Anal. Discrete Math.} {\bf 11} (2017), 81--107.
		\bibitem{NikiforovTrees}
		V.~Nikiforov, R.~Rojo, K.~Rooney and I.~Sciriha,
		On the $A_\alpha$-spectra of trees,
		\emph{Linear Algebra Appl.} \textbf{520} (2017), 286--305.
		
		\bibitem{bilalcoo} S. Pirzada,  Bilal A. Rather, R. U. Shaban and T. A. Chishti, Signless Laplacian eigenvalues of the zero-divisor graph associated to finite commutative ring $ \mathbb{Z}_{p^{M_{1}} q^{M_{2}}} $, \emph{Comm. Comb. Optim.} \textbf{8}(3) (2023) 561--574.
		\bibitem{bilalakcej} S. Pirzada,   Bilal A. Rather, H. A. Ganie and R. U. Shaban, On $ \alpha $-adjacency energy of graphs,  \emph{AKCE Int. J. Graphs Comb.} (2021) \textbf{18}(1) 2020 39--46.
		\bibitem{bilalejgta} S. Pirzada,  Bilal A. Rather, T. A. Chishti and U. Samme, On normalized Laplacian spectrum of zero-divisor graphs of commutative ring  $ \mathbb{Z}_{n} $, \emph{Elec. J. Graph Theory  Appl.} \textbf{9}(2) (2021) 331--345.
		\bibitem{bilalaims}	Bilal A. Rather, F. Ali, N. Ullah, Al-S. Mohammad, A. Din and  Sehra, $ A_{\alpha} $ matrix of commuting graphs of non-abelian groups, \textit{Aims Math.} \textbf{7}(8) (2022) 15436--15452.
		\bibitem{bilaljaa}	  Bilal A. Rather, H. A. Ganie and S. Pirzada, On the $ A_{\alpha} $--spectrum of joined union and its applications to power graphs of certain finite groups,  \emph{J. Algebra  Appl.} \textbf{22}(12) (2023) 2350257.
		\bibitem{bilalDlsurvey} Bilal A. Rather and Mustapha Aouchiche, Distance Laplacian spectra of graphs: a survey, \emph{Discrete Applied Mathematics} \textbf{361} (2025) 136--195.
		\bibitem{bilalijpam} Bilal A. Rather, M. Aouchiche and M. Imran,  On Laplacian integrability of comaximal graphs of commutative rings, \emph{Indian J. Pure Appl. Math.} \textbf{55}  (2024) 310--324. 
	
		\bibitem{so-integral}
		W.~So,
		\newblock Integral circulant graphs,
		\newblock {\em Discrete Math.} {\bf 306} (2006), 153--158.
		
		\bibitem{my} M. Young, Adjacency matrices of zero-divisor graphs of integer modulo $ n $, \emph {Involve} {\bf 8} (2015) 753--761.
		
		
		
		
	\end{thebibliography}
\end{document}